\documentclass[final]{zzz}
\usepackage{amsmath,amssymb,amsthm,amsfonts}
\usepackage{verbatim}
\usepackage{showkeys} 
\usepackage{enumerate}
\usepackage{soul}
\usepackage{xcolor}
\usepackage{hyperref}
\usepackage{faktor}
\usepackage{bm}
\usepackage{comment}
\usepackage{mathrsfs}

\hypersetup{
  colorlinks = true,
  urlcolor = blue,
  linkcolor = black,
  citecolor = red
}

\newcommand{\hyp}[5]{{}_{#1}F_{#2}\!\left(\genfrac{}{}{0pt}{}{#3}{#4};#5\right)}

\newcommand{\Ohyp}[5]{\,\mbox{}_{#1}{\bm{F}}_{#2}\!\left(
\genfrac{}{}{0pt}{}{#3}{#4};#5\right)}

\definecolor{Mycolor2}{HTML}{e85d04}
\sethlcolor{black}

\newcommand{\N}{{\mathbb N}}

\newcommand{\Hi}{{\mathbf H}}

\def\CH2{\CC\Hi^2_R}

\newcommand{\expe}{{\mathrm e}}
\newcommand{\dd}{{\mathrm d}}

\newcommand{\Z}{{\mathbb Z}}

\newcommand{\CC}{{{\mathbb C}}}
\newcommand{\CCast}{{{\mathbb C}^\ast}}

\newtheorem{thm}[lemma]{Theorem}
\newtheorem{cor}[lemma]{Corollary}
\newtheorem{rem}[lemma]{Remark}
\newtheorem{lem}[lemma]{Lemma}

\makeatletter
\def\eqnarray{\stepcounter{equation}\let\@currentlabel=\theequation
\global\@eqnswtrue
\tabskip\@centering\let\\=\@eqncr
$$\halign to \displaywidth\bgroup\hfil\global\@eqcnt\z@
 $\displaystyle\tabskip\z@{##}$&\global\@eqcnt\@ne
 \hfil$\displaystyle{{}##{}}$\hfil
 &\global\@eqcnt\tw@ $\displaystyle{##}$\hfil
 \tabskip\@centering&\llap{##}\tabskip\z@\cr}

\def\endeqnarray{\@@eqncr\egroup
   \global\advance\c@equation\m@ne$$\global\@ignoretrue}

\def\@yeqncr{\@ifnextchar [{\@xeqncr}{\@xeqncr[5pt]}}
\makeatother

\def\sideremark#1{\ifvmode\leavevmode\fi\vadjust{\vbox to0pt{\vss
 \hbox to 0pt{\hskip\hsize\hskip1em%
 \vbox{\hsize2cm\tiny\raggedright\pretolerance10000
 \noindent #1\hfill}\hss}\vbox to8pt{\vfil}\vss}}}

\newcommand{\C}{\mathbb{C}}

\numberwithin{equation}{section}

\begin{document}

\renewcommand{\PaperNumber}{***}

\FirstPageHeading

\ShortArticleName{
Integral representation for a product of two Jacobi functions of the second kind
}

\ArticleName{Integral representation for a product of two Jacobi functions of the second kind
}
\Author{Howard S.~Cohl\,$^\ast\orcidB{}$ 
 and Loyal Durand\,$^\dag\orcidC{}$
 \footnote{LD would like to thank the Aspen Center for Physics, supported by The National Science Foundation grant PHY-2210452, for its hospitality while parts of this work were done.}}

\AuthorNameForHeading{H.~S.~Cohl, 
L.~Durand}

\Address{$^\ast$~Applied and Computational Mathematics Division,
National Institute of Standards and Technology,
Gaithersburg, MD 20899-8910, USA
} 
\URLaddressD{
\href{http://www.nist.gov/itl/math/msg/howard-s-cohl.cfm}
{http://www.nist.gov/itl/math/msg/howard-s-cohl.cfm}
}
\EmailD{howard.cohl@nist.gov} 

\Address{$^\dag$~{Department of Physics, University of Wisconsin-Madison, Madison, WI 53706, USA,\\
\phantom{X}415 Pearl Court, Aspen, CO 81611, USA}
} 
\URLaddressD{
\href{http://hep.wisc.edu/~ldurand}
{http://hep.wisc.edu/$\sim$ldurand}
}
\EmailD{{ldurandiii@comcast.net}} 

\ArticleDates{Received ???, in final form ????; Published online ????}

\Abstract{By starting with Durand's double integral representation for a product of two Jacobi functions of the second kind, we derive an integral representation for a product of two Jacobi functions of the second kind in kernel form. We also derive a Bateman-type sum for a product of two Jacobi functions of the second kind. From this integral representation we derive integral representations for the Jacobi function of the first kind in both the hyperbolic and trigonometric contexts. From the integral representations for Jacobi functions, we also derive integral representations for products of limiting functions such as associated Legendre functions of the first and second kind, Ferrers functions and also Gegenbauer functions of the first and second kind. By examining the behavior of one of these products near singularities of the relevant functions, we also derive integral representations for single functions, including a Laplace-type integral representation for the Jacobi function of the second kind. Finally, we use the product formulas for the functions of the second kind to derive Nicholson-type integral relations for the sums of squares of Jacobi functions of the first and second kinds, and in a confluent limit, Laguerre functions of the first and second kinds, which generalize the relation $\expe^{ix}\expe^{-ix}=1$ to those functions.}


\Classification{33C05, 33C45, 53C22, 53C35}




\section{Introduction}
\label{intro}

Product formulas for two Jacobi functions of the first kind $P_\lambda^{(\alpha,\beta)}(z_1)P_\lambda^{(\alpha,\beta)}(z_2)$  have played an important role in the theory of Jacobi series. Gasper \cite{Gasper71} derived an expression for the product of two Jacobi polynomials in the kernel form
\begin{equation}
\label{Gasper_RR}
R_n^{(\alpha,\beta)}(x_1)R_n^{(\alpha,\beta)}(x_2) = \int_{-1}^1R_n^{(\alpha,\beta)}(x_3)K(x_1,x_2,x_3)(1-x_3)^\alpha(1+x_3)^\beta \dd x_3
\end{equation}
with $R_n^{(\alpha,\beta)}(x) = P_n^{(\alpha,\beta)}(x) /P_n^{(\alpha,\beta)}(1)$, obtained an integral representation for the kernel $K$, and established its positivity for $\alpha\geq\beta\geq-1/2$ with $\alpha>-1/2$. This was used to establish a convolution structure with positive kernel for Jacobi series. Gasper extended his analysis in \cite{Gasper1972} to obtain explicit results for the kernel in various ranges of $\alpha,\,\beta$,  investigated their properties in detail, and connected the results to an associated Banach algebra.

Koornwinder obtained the product formula as a double integral in  \cite{Koornwinder1972AI} from his addition formula for Jacobi polynomials.  Flensted-Jensen and Koornwinder later gave a completely analytic derivation of that formula, reduced it to kernel form, and used the results to further study the convolution product and Banach algebra in  \cite{FlenstedJensenKoornwinder73}. That paper also included an extension of the formula to the case of mixed products  of  functions of the first and second kind $P_\lambda^{(\alpha,\beta)}(z_1)Q_\lambda^{(\alpha,\beta)}(z_2)$ \cite[p.~255]{FlenstedJensenKoornwinder73}.  The product formula for functions of the first kind was used by those authors to give a new, completely analytic derivation of Koornwinder's addition theorem for general values of the parameters \cite[(2.13)]{FlenstedJensenKoorn79}  and establish the positivity of the inverse Fourier-Jacobi transform. 
 
In \cite[(13)]{Durand75} Durand derived an integral representation for the product of two Gegenbauer functions $D_{\lambda-\nu}^{\alpha+\nu}(z)$ of the second kind starting with a formula for the product $D_{\lambda-\nu}^{\alpha+\nu}(z_1)C_{\lambda-\nu}^{\alpha+\nu}(z_2)$ of functions of the second and first kinds. This followed from an addition formula for a function $D_\lambda^\alpha(Z)$ of the second kind with an argument $Z:=z_1z_2+t\sqrt{z_1^2-1}\sqrt{z_2^2-1}$ which was  was obtained through a modification of Gegenbauer's addition formula for $C_\lambda^\alpha(Z)$ \cite[(8.6)]{DurandFishSim}. The product formula 
\begin{eqnarray}
&&\hspace{-1.0cm}D_{\lambda-\nu}^{\alpha+\nu}(z_1)D_{\lambda-\nu}^{\alpha+\nu}(z_2) = \frac{\expe^{i\pi(\alpha+2\nu)}\Gamma(2\alpha-1)\Gamma(\nu+1)\Gamma(\lambda+2\alpha)}{2^{2\alpha+2\nu-1}[\Gamma(\alpha+\nu)]^2\Gamma(\nu+2\alpha-1)\Gamma(\lambda-\nu+1)} [(z_1^2-1)(z_2^2-1)]^{-\nu/2} \nonumber \\
\label{DDproduct}
&&\hspace{2.7cm}\times  \int_1^\infty D_\lambda^\alpha\left(z_1z_2+t\sqrt{z_1^2-1}\sqrt{z_2^2-1}\right)C_\nu^{\alpha-\frac{1}{2}}(t)\left(t^2-1\right)^{\alpha-1}\,\dd t,
\end{eqnarray}
is easily put in kernel form by the change of variable from $t$ to $z_3=z_1z_2+t\sqrt{z_1^2-1}\sqrt{z_2^2-1}$; see \eqref{DDkernel} below. It was used in its original form in \cite{Durand75} to derive Nicholson-type integrals for the Gegenbauer and associated Legendre functions and to investigate various properties of those functions and the Bessel and Hermite functions obtained as confluent limits of the Gegenbauer functions. 

The proof of \eqref{DDproduct} was described in \cite{Durand75}. It began with the addition formula for $D_\lambda^\alpha(Z)$ in \cite[(8.6)]{DurandFishSim} which required $z_1>z_2$ for convergence. Projecting out the $n^{\rm th}$ term in the series using the orthogonality relation for Gegenbauer polynomials gave the product formula 
\begin{eqnarray}
D_{\lambda-\nu}^{\alpha+\nu}(z_1)C_{\lambda-\nu}^{\alpha+\nu}(z_2) &=&  \frac{\Gamma(2\alpha-1)\Gamma(\nu+1)\Gamma(\lambda+2\alpha)}{2^{2\alpha+2\nu-1}[\Gamma(\alpha+\nu)]^2\Gamma(\nu+2\alpha-1)\Gamma(\lambda-\nu+1)}[(z_1^2-1)(z_2^2-1)]^{-\nu/2}  \nonumber \\
\label{CDproduct}
&&\times\int_{-1}^1D_\lambda^\alpha\left(z_1z_2+t\sqrt{z_1^2-1}\sqrt{z_2^2-1}\right)C_\nu^{\alpha-\frac{1}{2}}(t)\left(1-t^2\right)^{\alpha-1}\dd t.
\end{eqnarray}
Opening the integration contour in the upper half plane in $t$ allowed the integral to be written as an integral from $-1$ to $-\infty$ and an integral from $+\infty$ to 1, connected by a loop at $\infty$ on which the integrand vanished. The integral on the right-hand segment was symmetric in $z_1,\,z_2$ and had the asymptotic behavior of a product of two $D$s for $z_1,\,z_2\rightarrow \infty$. The integrals on the two segments were independent and satisfied the Gegenbauer equations separately in $z_1,\,z_2$. A check on the normalization for $z_1,\,z_2\rightarrow\infty$ then established the result in \eqref{DDproduct}. Note that the integrals in \eqref{CDproduct} and \eqref{DDproduct}  differ by a transition from  a trigonometric variable $t=\cos{\theta}\in(-1,1)$ in \eqref{CDproduct} to a hyperbolic variable $t=\cosh{\theta}\in (1,\infty)$ in \eqref{DDproduct}. The coefficients on the right-hand sides of the expressions are the same except for the phase factors that arise from the factors $[(z_1^2-1)(z_2^2-1)]^{-n/2}$ and $(1-t^2)^\alpha$ in this transition.  

This work was extended in  \cite[(2.5)]{Durand78} to obtain a double-integral representation for a product of two Jacobi functions of the second kind. That derivation started with the extension of Koornwinder's double-sum addition theorem for Jacobi polynomials \cite{Koornwinder1972AI,Koornwinder74,Koornwinder75}  to general Jacobi functions $P_\lambda^{(\alpha,\beta)}(Z)$ with argument 
\begin{eqnarray}
\label{Z1}
&&\hspace*{0cm} Z=z_1z_2+\sqrt{z_1^2-1}\sqrt{z_2^2-1}\,rt+\frac{1}{2}(z_1-1)(z_2-1)(r^2-1), \quad r\in[0,1],\quad t\in[-1,1] ,\\
\label{Z2}
&&\hspace*{0.5cm} = 2y_1^2y_2^2-1+4y_1y_2\sqrt{y_1^2-1}\sqrt{y_2^2-1}\,rt+2(y_1^2-1)(y_2^2-1)r^2, \quad z_i=2y_i^2-1,
\end{eqnarray}
given in \cite{FlenstedJensenKoorn79}.  This was used to obtain an addition formula for a Jacobi function of the second kind $Q_\lambda^{(\alpha,\beta)}(Z)$ with identical structure. A  product formula for a Jacobi $P$ and $Q$ followed from that addition formula by projecting out a single term using the orthogonality properties of the Jacobi and Gegenbauer polynomials of arguments $2r^2-1$ and $t$ that appear in the sum. It had also been derived analytically in \cite{FlenstedJensenKoornwinder73}. 

Neither the $PQ$ addition formula for $Q_\lambda^{(\alpha,\beta)}(Z)$ nor the associated product formula for a Jacobi $P$ and a $Q$ were given in \cite{Durand78}. These have now been given with proofs and extensive applications in \cite[\S4.1(4.2) and \S4.1(4.15)]{CohletaladditionPQ} respectively. 

The $PQ$ product formula was used to derive a $QQ$ product using the same method as sketched above for the Gegenbauer $DD$ formula.  The integration contours in $t$ and $r$ were opened up in their respective upper half planes to obtain expressions that involved  separate integrations on the intervals $(-\infty,-1]$ and $(\infty,1]$ in $t$, and $(0,-\infty)$ and $(\infty,1]$ in $r$. The product of the integrals on the right-hand intervals was again independent of the other terms in the full expression, was symmetric in $z_1,\,z_2$, and had the asymptotic form for the product of two Jacobi functions of the second kind with those arguments, and gave the $QQ$ product. The result was used  to obtain the corresponding double-integral product formulas for Gegenbauer, Laguerre, Bessel, and Hermite functions of the second kind \cite[(2.5)]{Durand78}, and later   to obtain a remarkable addition formula for the Jacobi functions of the second kind \cite[(36), (38)]{Durand79}.

The addition and product formulas noted above have a deep connection to the theory of Lie groups and their representations. Thus the Gegenbauer polynomials can be interpreted for $\lambda$ and $2\alpha$ integers and $-1\leq z_1,\,z_2,\,t\leq 1$ in terms of spherical functions on the hypersphere ${\bf S}^{2\alpha+1}$ which are invariant under the subgroup SO$(2\alpha+1)$ of SO$(2\alpha+2)$ \cite[Chapter 9]{Vilen}. Gegenbauer's addition theorem follows from the decomposition of rotations on SO(2$\alpha$+2)/SO(2$\alpha$+1) in terms of those functions.  

The group-theoretical background for the addition formula for the Jacobi polynomials is discussed in \cite{Koornwinder1972AI} and in much more detail in \cite{Koornwinder73}. Those polynomials can be interpreted among other ways as spherical functions on SO$(p+q)$ which are invariant under the actions of the subgroup SO$(p)\times$SO$(q)$, and give unitary representations of the group with the usual orthogonality and completeness relations.

The addition formulas for general values of $\alpha$ and $\beta$ follow from analytic extensions of the results for $2\alpha,\,2\beta\in\N_0$. The underlying connection to Lie groups appears through the differential recurrence relations for the general functions. These provide realizations of the corresponding Lie algebras  when acting on either functions of the first or second kind. Exponentiation of those operators gives group operators which act on those functions and can be used to explore their properties \cite{Durand2003_II}. The differential equations satisfied by the functions also follow from the quadratic Casimir invariants of the groups.

The natural setting for the functions of the second kind is in hyperbolic rather than hyperspherical spaces. In the case of the Gegenbauer functions, we transfer from ${\bf S}^{2\alpha+1}$, $2\alpha=n$ integer for the $C$s with arguments $x = \cos{\theta}\in[-1,1]$ where the functions are analytic, to ${\bf H}^{2\alpha+1}$ for $D$s with arguments $z=\cosh{\theta}\in(1,\infty)$ where the the functions are again analytic, with the associated group SO$(2\alpha+1,1)$. In this setting, as is familiar for Mehler--Fock transforms \cite[\S\href{http://dlmf.nist.gov/14.20.vi}{14.20(vi)}]{NIST:DLMF}, projections involving Gegenbauer or Jacobi functions of integer degree with argument $x\in(-1,1]$ are replaced by integrals of those functions of argument $z$ over  the infinite interval $z\in(1,\infty)$, while the sums over functions of integer degree characteristic of eigenfunction expansions are replaced by integrals over complex degrees. The structure of the expansions and their inverses and their group-connected coefficients are otherwise unchanged. Thus the  product formula for two Gegenbauer functions of the second kind in \eqref{DDproduct} is identical in form to that for two $C$s given in \cite[(3.15.1.20)]{Erdelyi},
\begin{eqnarray}
&&\hspace{-1.0cm}C_{\lambda-\nu}^{\alpha+\nu}(z_1)C_{\lambda-\nu}^{\alpha+\nu}(z_2) = \frac{\Gamma(2\alpha-1)\Gamma(\nu+1)\Gamma(\lambda+2\alpha)}{2^{2\alpha+2\nu-1}[\Gamma(\alpha+\nu)]^2\Gamma(\nu+2\alpha-1)\Gamma(\lambda-\nu+1)(z_1^2-1)^{\nu/2}(z_2^2-1)^{\nu/2}} \nonumber \\
\label{CCproduct}
&&\hspace{3.5cm}\times  \int_1^\infty C_\lambda^\alpha\left(z_1z_2+t\sqrt{z_1^2-1}\sqrt{z_2^2-1}\right)C_\nu^{\alpha-\frac{1}{2}}(t)\left(1-t^2\right)^{\alpha-1}\,\dd t,
\end{eqnarray}
$z_1,\,z_2\in(-\infty,1]$, with the replacement of spherical variables  by hyperbolic variables. In  particular, the leading numerical factors determined by the underlying Lie algebra are the same. This also holds as expected in the case of the product formulas for Jacobi functions.

The product formulas are in some ways more fundamental than the addition formulas for functions of the first kind from which they were initially derived.  In the case of the Jacobi functions, Flensted-Jensen and Koornwinder \cite{FlenstedJensenKoornwinder73} derived the product formula analytically and used the result to derive Koornwinder's addition theorem, while Durand  used that approach to derive his addition formula for Jacobi functions of the second kind \cite{Durand79} from his product formula. He also attempted to put that product formula in the kernel form derived by Flensted-Jensen and Koornwinder in the case of the product of two functions of the first kind, There is an error in the resulting formula which we correct here.

We begin the paper in Section \ref{sec:prelim} with the review and development of the many relations among the Gegenbauer, Legendre, and Jacobi functions.
In Section \ref{sec:double_int} we put Durand's $QQ$ product formula in  kernel form. We then use the result in the following sections to obtain corresponding results for double products of Legendre and Gegenbauer functions, then through a limiting procedure, integral representations for single functions,
and conclude with a discussion of the Nicholson-type expressions for the sum of squares of Legendre, Gegenbauer, Laguerre, Bessel, and Hermite functions of the first and second kinds that follow from the product formulas.

\section{Mathematical preliminaries \label{sec:prelim}}

Please observe the following simple results from the theory of integration.
For some $\alpha\in\CC$, $0<\epsilon\ll 1$, $x_0\gg 1$ the integrals $\int_0^\epsilon x^\alpha\,\dd x$, 
$\int_{x_0}^\infty x^\alpha\,\dd x$ exist if $\Re\alpha>-1$ and $\Re\alpha<-1$ respectively. These results are important when one is considering the allowed values of parameters involved in certain integrals which we will derive below.
Define the regularized Gauss hypergeometric function as
\cite[\href{http://dlmf.nist.gov/15.2.E2}{(15.2.2)}]{NIST:DLMF}
\begin{equation}
\Ohyp21{a,b}{c}{z}:=\frac{1}{\Gamma(c)}\hyp21{a,b}{c}{z}=\sum_{n=0}^\infty \frac{(a)_n(b)_n}{\Gamma(c+n)}\frac{z^n}{n!}.
\end{equation}

\subsection{Associated Legendre and Gegenbauer functions}

The functions which satisfy quadratic transformations of the Gauss 
hypergeometric function are given by Gegenbauer and associated Legendre functions of the  
first and second kind. As we will see, these
functions correspond to Jacobi functions of the first and second kind when 
their parameters satisfy certain relations. We now describe some of the 
properties of these functions which have a deep and long history.

The Gegenbauer (ultraspherical) polynomial which is an important
specialization of the Jacobi polynomial for
symmetric parameter values is also given
in terms of a terminating Gauss hypergeometric function
\cite[\href{http://dlmf.nist.gov/18.7.E1}{(18.7.1)}]{NIST:DLMF}
\begin{equation}\label{GegJac}
C_n^\alpha(z)=\frac{(2\alpha)_n}{(\alpha+\frac12)_n}
P_n^{(\alpha-\frac12,\alpha-\frac12)}(z)
=\frac{\sqrt{\pi}\,\Gamma(2\alpha+n)}
{2^{2\alpha-1}\Gamma(\alpha)n!}
\Ohyp21{-n,2\alpha+n}{\alpha+\frac12}{\frac{1-z}{2}}.
\end{equation}

Gegenbauer functions which generalize Gegenbauer polynomials
to complex degree are solutions $w=w(z)=w_\lambda^\alpha(z)$ to the 
Gegenbauer differential equation 
\cite[\href{http://dlmf.nist.gov/18.8.T1}{Table 18.8.1}]{NIST:DLMF}
\begin{equation}
(z^2-1)\frac{\dd^2w(z)}{\dd z^2}+\left(2\lambda+1\right)z\frac{\dd w(z)}{\dd z}
-\lambda(\lambda+2\alpha)w(z)=0.
\label{Gende}
\end{equation}
The Gegenbauer function of the first kind
is defined as \cite[\href{http://dlmf.nist.gov/15.9.E15}{(15.9.15)}]{NIST:DLMF}
\begin{equation}\label{Gegenbauerfuncdef}
C_\lambda^\alpha(z):=
\frac{\sqrt{\pi}\,\Gamma(2\alpha+\lambda)}
{2^{2\alpha-1}\Gamma(\alpha)\Gamma(\lambda+1)}
\Ohyp21{-\lambda,2\alpha+\lambda}{\alpha+\frac12}{\frac{1-z}{2}},
\end{equation}
and it is the clear extension of the Gegenbauer
polynomial when the index is allowed to be 
a complex number as well as a non-negative integer. 
The Gegenbauer function of the first kind is related
to the associated Legendre function 
of the first kind \cite[\href{http://dlmf.nist.gov/14.3.E6}{(14.3.6)} and \href{http://dlmf.nist.gov/14.21.i}{\S14.21(i)}]{NIST:DLMF}
\begin{equation}
P_\nu^\mu(z):=
\left(\frac{z+1}{z-1}\right)^{\frac12\mu}
\Ohyp21{-\nu,\nu+1}{1-\mu}{\frac{1-z}{2}},
\label{associatedLegendrefunctionP}
\end{equation}
where $|1-z|<2$, and elsewhere in $z$ by analytic continuation. One has an expression for the leading term in an asymptotic expansion of the associated Legendre function of the first kind as $|z|\to\infty$, namely
\begin{equation}
P_\nu^{-\mu}(z)\sim\frac{2^\nu\,\Gamma(\nu+\frac12)z^\nu}{\sqrt{\pi}\,\Gamma(\nu+\mu+1)},
\end{equation}
which follows from \cite[Entry 3]{MOS} and the Gauss sum \cite[\href{http://dlmf.nist.gov/15.4.E20}{(15.4.20)}]{NIST:DLMF} and requires $\Re(\nu)>-\frac12$.
The associated Legendre function of the first kind 
is also related
to the Gegenbauer function of the first kind,
\cite[\href{http://dlmf.nist.gov/14.3.E22}{(14.3.22)}]{NIST:DLMF}
\begin{equation}
P_\nu^{-\mu}(z)=\frac{\Gamma(2\mu+1)
\Gamma(\nu-\mu+1)}{2^\mu\Gamma(\nu+\mu+1)\Gamma(\mu+1)}(z^2-1)^{\frac12\mu}
C_{\nu-\mu}^{\mu+\frac12}(z),
\label{FerPtoGegC}
\end{equation}
or equivalently
\begin{eqnarray} \label{ClmLegP}
\hspace{-9.5cm}C_\lambda^\alpha(z)=\dfrac{\sqrt{\pi}\,\Gamma(2\alpha+\lambda)}
{2^{\alpha-\frac12}\Gamma(\alpha)\Gamma(\lambda+1)}
\frac{P_{\lambda+\alpha-\frac12}^{\frac12-\alpha}(z)}
{(z^2-1)^{\frac{\alpha}{2}-\frac14}}.
\end{eqnarray}

The associated Legendre function of the second kind
$Q_\nu^\mu:\C\setminus(-\infty,1]\to\C,$ $\nu+\mu\notin-\N$, can be 
defined in terms of the Gauss hypergeometric function as follows 
\cite[\href{http://dlmf.nist.gov/14.3.E7}{(14.3.7)} and \href{http://dlmf.nist.gov/14.21.i}{\S14.21}]{NIST:DLMF}
\begin{equation}\label{for:lf2k}
Q_\nu^\mu(z):=\frac{\sqrt{\pi}\expe^{i\pi\mu}\Gamma(\nu+\mu+1)
(z^2-1)^{\frac12\mu}}{2^{\nu+1}
z^{\nu+\mu+1}}\Ohyp21{\frac{\nu+\mu+1}{2},
\frac{\nu+\mu+2}{2}}{\nu+\frac32}{\frac{1}{z^2}},
\end{equation}
for $|z|>1$ and elsewhere in $z$ by analytic continuation
of the Gauss hypergeometric function.
One may also define the associated Legendre function of the second kind
using cf.~\cite[entry 24, p.~161]{MOS}
\begin{equation}
 Q_\nu^\mu(z):=
\frac{2^\nu \expe^{i\pi\mu}\Gamma(\nu+1)\Gamma(\nu+\mu+1) (z+1)^{\frac12\mu}}
{
(z-1)^{\frac12\mu+\nu+1}}
\Ohyp21{
\nu+1,\nu+\mu+1}{
2\nu+2}{\frac{2}{1-z}}.
\label{Qdefntwodivide1mz}
\end{equation}
Note that the associated Legendre function of the second kind satisfies the following parity relation in the order $\mu$, namely cf.~\cite[\href{http://dlmf.nist.gov/14.9.E14}{(14.9.14)}]{NIST:DLMF},
\begin{equation}
Q_{\nu}^{-\mu}(z)=\expe^{-2\pi i\mu}\frac{\Gamma(\nu-\mu+1)}{\Gamma(\nu+\mu+1)}Q_\nu^\mu(z).
\label{ptyQmu}
\end{equation}
The Gegenbauer function of the second kind 
$D_\lambda^\alpha(z)$ is the second
linearly independent solution to the Gegenbauer differential equation 
(see \cite[\S1]{DurandFishSim}).
Two representations
which will be useful for us in comparing to
the Jacobi function of the second kind are given by
\begin{eqnarray}
&&\hspace{-4.8cm}D_\lambda^\alpha(z)=
\frac{\expe^{i\pi\alpha}2^{\lambda}\Gamma(\lambda+\alpha+\frac12)
\Gamma(\lambda\!+\!2\alpha)}{\sqrt{\pi}\Gamma(\alpha)(z\!-\!1)^{\lambda+\alpha+\frac12}
(z\!+\!1)^{\alpha-\frac12}}\Ohyp21{\lambda\!+\!1,\lambda\!+\!\alpha\!+\!\frac12}
{2\lambda\!+\!2\alpha\!+\!1}{\frac{2}{1\!-\!z}}\\[0.05cm]
&&\hspace{-4.75cm}\hspace{1.1cm}=\frac{\expe^{i\pi\alpha}\Gamma(\lambda+2\alpha)}{\Gamma(\alpha)(2z)^{\lambda+2\alpha}}\Ohyp21{\frac12\lambda+\alpha,\frac12\lambda+\alpha+\frac12}{\lambda+\alpha+1}{\frac{1}{z^2}}.
\label{GegDhyper}
\end{eqnarray}
The Gegenbauer function of the second kind is clearly related to the associated 
Legendre function of the second kind and is given by \cite[(2.4)]{DurandFishSim}
\begin{eqnarray}
&&\hspace{-6.7cm}D_\lambda^\alpha(z):=\frac{\expe^{2\pi i(\alpha-\frac14)}\Gamma(\lambda+2\alpha)}
{\sqrt{\pi}\,2^{\alpha-\frac12}\Gamma(\alpha)\Gamma(\lambda+1)
(z^2-1)^{\frac12\alpha-\frac14}}Q_{\lambda+\alpha-\frac12}^{\frac12-\alpha}(z)\\
&&\hspace{-5.5cm}=\frac{\expe^{\frac{i\pi}{2}}}
{\sqrt{\pi}\,2^{\alpha-\frac12}\Gamma(\alpha)
(z^2-1)^{\frac12\alpha-\frac14}}Q_{\lambda+\alpha-\frac12}^{\alpha-\frac12}(z)
\label{DlmLegQ}
\end{eqnarray}
Therefore, one also has 
\begin{eqnarray}
&&\hspace{-7.2cm}Q_\nu^\mu(z)=\frac{\expe^{2\pi i(\mu-\frac14)}\sqrt{\pi}\,\Gamma(\frac12-\mu)\Gamma(\nu+\mu+1)}{2^\mu\Gamma(\nu-\mu+1)(z^2-1)^{\frac12\mu}}D_{\nu+\mu}^{\frac12-\mu}(z)\\
&&\hspace{-6.1cm}=\expe^{-\frac{i\pi}{2}}2^\mu\sqrt{\pi}\,\Gamma(\mu+\tfrac12)(z^2-1)^{\frac12\mu}
D_{\nu-\mu}^{\mu+\frac12}(z).
\label{QnmtoDla}
\end{eqnarray}
\begin{rem}
Since the Gegenbauer functions of the first and second kind
are directly proportional to the associated 
Legendre functions of the first and second kind, all of
the Gauss hypergeometric representations which
exist for the associated Legendre functions (see e.g., \cite[\S3.2]{Erdelyi}, \cite[\S4.1.2]{MOS}),  also exist 
for the Gegenbauer functions.
\end{rem}

\begin{rem}
The Gegenbauer functions of the first and second kind satisfy the same recurrence relations (see e.g., \cite[\S4]{DurandFishSim}, \cite[\href{https://dlmf.nist.gov/18.9}{\S18.9}]{NIST:DLMF}), as is easily shown using the integral representations for those functions given in \cite[(1.4), (1.6)]{DurandFishSim}. There are additional differential recurrence relations which can be obtained from the relations for the Legendre functions derived in \cite{Durand2003_II}.
 These differential recurrence relations reflect the action of operators in the underlying Lie algebras.
\end{rem}

\noindent The associated Legendre function of the second kind has the following special value.
\begin{lem}
\label{lem3}Let $z\in\CCast\setminus(-\infty,1]$, $\mu\in\CC\setminus-\N_0$. Then
\begin{equation}
Q_{\mu-1}^{\mu}(z)=\frac{\expe^{i\pi\mu}2^{\mu-1}\Gamma(\mu)}{(z^2-1)^{\frac12\mu}}.
\end{equation}
\end{lem}
\begin{proof}
This result follows directly by specializing the following Gauss hypergeometric representation of the associated Legendre function of the second kind \cite[Entry 26]{MOS}, namely
\begin{equation}
Q_\nu^\mu(z)=\frac{\expe^{i\pi\mu}\sqrt{\pi}\,\Gamma(\nu+\mu+1)}{2^{\nu+1}(z^2-1)^{\frac12(\nu+1)}}\Ohyp21{\frac12(\nu-\mu+1),\frac12(\nu+\mu+1)}{\nu+\frac32}{\frac{1}{1-z^2}}.    
\end{equation}
This completes the proof.
\end{proof}

\noindent The associated Legendre function of the second kind has the following behavior near the singularity at $z\sim 1$, namely, as $\epsilon\to 0^{+}$
\begin{eqnarray}
&&\hspace*{-6.7cm}Q_\nu^{\mu}(1+\epsilon)\sim\frac{\expe^{i\pi\mu}\sqrt{\pi}\,\Gamma(-\mu)\Gamma(\nu+\mu+1)\epsilon^{\frac12\mu}}{2^{\frac12\mu+1}\sqrt{\pi}\,\Gamma(\nu-\mu+1)}, \quad \Re \mu<0, \label{asympQ1mp_2} \\
&&\hspace*{-6.7cm}Q_\nu^\mu(1+\epsilon)\sim 2^{\frac12\mu-1}\expe^{i\pi\mu}\Gamma(\mu)\epsilon^{-\frac12\mu}, \quad \Re\mu>0.
\label{asympQ1mp}
\end{eqnarray}
which follow from \eqref{Qdefntwodivide1mz} by taking $z\sim 1+\epsilon$ and then using transformation to the inverse variable \cite[\href{http://dlmf.nist.gov/15.8.E2}{(15.8.2)}]{NIST:DLMF}.

\medskip

\noindent One has the following two transformations for the Gegenbauer function of the first kind, see \cite[(3.1), (3.2)]{DurandFishSim}, namely
\begin{eqnarray}
&&\hspace{-8.6cm}C_\lambda^\alpha(z)=-\frac{\sin(\pi\lambda)}{\sin(\pi(\lambda+2\alpha))}C_{-\lambda-2\alpha}^\alpha(z),
\label{GegCtran1}\\
&&\hspace{-8.6cm}C_\lambda^\alpha(z)=\frac{\expe^{-i\pi\alpha}\sin(\pi\lambda)}{\sin(\pi(\lambda+\alpha))}\left(D_\lambda^\alpha(z)-D_{-\lambda-2\alpha}^\alpha(z)\right).
\end{eqnarray}
Also, note that \eqref{QnmtoDla} implies the following transformation for the Gegenbauer function of the second kind
\begin{equation}
D_{\lambda}^\alpha(z)=\frac{\expe^{2i\pi(\alpha+\frac12)}\Gamma(1-\alpha)\Gamma(\lambda+2\alpha)}{2^{2\alpha-1}\Gamma(\lambda+1)\Gamma(\alpha)(z^2-1)^{\alpha-\frac12}}D_{\lambda+2\alpha-1}^{1-\alpha}(z),
\end{equation}
which is equivalent to
\begin{equation}
\hspace{0.1cm}D_\lambda^{-\alpha}(z)=\frac{\expe^{-2\pi i(\alpha+\frac12)}2^{2\alpha+1}\Gamma(\alpha+1)\Gamma(\lambda-2\alpha)(z^2-1)^{\alpha+\frac12}}{\Gamma(\lambda+1)\Gamma(-\alpha)}D_{\lambda-2\alpha-1}^{\alpha+1}(z).
\end{equation}

\noindent One also has the Whipple transformations for Gegenbauer functions of the first and second kinds (see e.g., \cite[\href{http://dlmf.nist.gov/14.9.E16}{(14.9.16)}, \href{http://dlmf.nist.gov/14.9.E17}{(14.9.17)}]{NIST:DLMF}, \cite[(3.3.1.13--14)]{Erdelyi}). The first Whipple transformation relates the Gegenbauer function of the first kind with the Gegenbauer function of the second kind, namely
\begin{equation}
C_\lambda^\alpha(z)=\frac{2^{\lambda+\alpha+\frac12}\expe^{-i\pi(\lambda+\alpha-\frac12)}\Gamma(-\lambda)\Gamma(\lambda+\alpha+\frac12)\sin(\pi\lambda)}{\sqrt{\pi}\,2^{\frac12\alpha-\frac14}\Gamma(\alpha)\,(z^2-1)^{\frac12(\lambda+\alpha)+\frac14}}
D_{-\lambda-1}^{\lambda+\alpha+\frac12}\left(\frac{z}{\sqrt{z^2-1}}\right).
\end{equation}
The second Whipple transformation relates the Gegenbauer function of the second kind to the Gegenbauer function of the first kind
\begin{equation}
D_{\lambda}^\alpha(z)=
\frac{2^\lambda\,\expe^{i\pi(\alpha-1)}\sqrt{\pi}\,\Gamma(\lambda+\alpha+\frac12)}{\Gamma(\lambda+1)\Gamma(\alpha)\sin(\pi\lambda)(z^2-1)^{\frac12\lambda+\alpha}}C_{-\lambda-1}^{\lambda+\alpha+\frac12}\left(\frac{z}{\sqrt{z^2-1}}\right).
\label{WhippleD}
\end{equation}
The Gegenbauer function of the first kind satisfies the following special value.
\begin{lem}
\label{lem1}
Let $z\in\CC\setminus(-\infty,-1]$, $\lambda\in\CCast$. Then
\begin{equation}
C_\lambda^{\frac12-\lambda}(z)=\frac{2^\lambda\sqrt{\pi}\,(z+1)^{\lambda}}{\Gamma(\lambda+1)\Gamma(\frac12-\lambda)}.
\end{equation}
\end{lem}
\begin{proof}
This can, for instance, be obtained by expressing the Gegenbauer function of the first kind in terms of an associated Legendre function of the first kind,
namely
\begin{equation}
C_\lambda^{\frac12-\lambda}(z)=\frac{2^\lambda\sqrt{\pi}\,\Gamma(1-\lambda)(z^2-1)^{\frac12\lambda}}{\Gamma(\lambda+1)\Gamma(\frac12-\lambda)}
P_0^\lambda(z),
\end{equation}
and then evaluating $P_0^\lambda$ using the
Gauss hypergeometric representation \eqref{associatedLegendrefunctionP}.
\end{proof}

\subsection{Ferrers functions and Gegenbauer functions in the trigonometric context}

Ferrers functions (associated Legendre functions of on-the-cut) satisfy the associated Legendre function differential equation and are analytically continued from $(-1,1)$. As we will see, for certain combinations
of the parameters the Jacobi functions 
of the first and second kind in the trigonometric context are related to associated Legendre 
functions of the first and second kind in the trigonometric context (Ferrers functions) 
and the Gegenbauer functions of the first and second kind in the trigonometric context.

\begin{rem}
The term `on-the-cut' simply refers to the functions carefully defined on the real line segment $(-1,1)$. Since these functions can be easily parametrized by some $x=\cos\theta$ with $\theta\in(0,\pi)$, we will refer to these as functions as being in the trigonometric context.
\end{rem}

\noindent The Ferrers functions are given by 
`averages' of values of the associated Legendre functions immediately 
above and below the segment $x\in(-1,1)$ and are given by \cite[\href{http://dlmf.nist.gov/14.23.E1}{(14.23.1)}, \href{http://dlmf.nist.gov/14.3.E2}{(14.23.2)}]{NIST:DLMF}
\begin{eqnarray}
\label{defFerP}
&&\hspace{-6.5cm}{\sf P}_\nu^\mu(x):=\expe^{\pm\frac12 i\pi\mu }
P_\nu^\mu(x\pm i0),\\ 
\label{defFerQ}
&&\hspace{-6.5cm}{\sf Q}_\nu^\mu(x):=\tfrac12 
\expe^{-i\pi\mu}\left(\expe^{-\frac12 i\pi\mu}Q_\nu^\mu(x+i0)
+\expe^{\frac12 i\pi\mu}Q_\nu^\mu(x-i0)\right).
\end{eqnarray}
The definitions for the Gegenbauer functions of the first and second 
kind in the trigonometric context are given by \cite[(3.3), (3.4)]{Durand78}.
They are given by averages of values of the Gegenbauer functions 
immediately above and below the cut $(-1,1)$ and are given by 
\begin{eqnarray}
\label{Ccutdef}
&&\hspace{-7.6cm}{\sf C}_\lambda^\alpha(x):=C_\lambda^\alpha(x\pm i0),\\
\label{Dcutdef}
&&\hspace{-7.6cm}{\sf D}_\lambda^\alpha(x):=
\expe^{-\frac{i\pi}{2}}D_\lambda^\alpha(x+i0)+
\expe^{-2i\pi(\alpha-\frac14)}D_\lambda^\alpha(x-i0).
\end{eqnarray}
While studying functions in the trigonometric context, it is useful to map from the hyperbolic context using the following rules \cite[Text after (8.3.4)]{Abra}, namely 
\begin{equation}
\left.
\begin{array}{lcl}
(z+1)&\mapsto &(1+x)\\[0.05cm]
(z-1)&\mapsto &\expe^{\pm i\pi}(1-x)\\[0.05cm]
(z^2-1)&\mapsto&\expe^{\pm i\pi}(1-x^2)
\end{array}\right\}.
\label{trighyprules}
\end{equation}
for $z=x\pm i0$, $x\in(-1,1)$.

\medskip
The Ferrers function of the first kind
$\mathsf{P}_\nu^\mu:\CC\setminus((-\infty,-1]\cup[1,\infty))\to\C$ has the following single Gauss hypergeometric representations
\cite[\href{http://dlmf.nist.gov/15.3.E1}{(14.3.1)}]{NIST:DLMF}
\begin{equation}
{\sf P}_\nu^\mu(x):=
\left(\frac{1+x}{1-x}\right)^{\frac12\mu}
\Ohyp21{-\nu,\nu+1}{1-\mu}{\frac{1\!-\!x}{2}},
\label{FerrersPdefnGauss2F1}
\end{equation}

\noindent The Ferrers function of the first kind ${\sf P}_\nu^{-\mu}$ is related to 
the Gegenbauer function of the first kind \cite[\href{http://dlmf.nist.gov/14.3.E21}{(14.3.21)}]{NIST:DLMF}
\begin{eqnarray} \label{relFerPGeg}
\hspace{-7.5cm}
{\sf P}_\nu^{-\mu}(x)&=&\frac{\Gamma(2\mu+1)\Gamma(\nu-\mu+1)}
{2^\mu\Gamma(\nu+\mu+1)\Gamma(\mu+1)}(1-x^2)^{\frac{\mu}{2}}
{\sf C}_{\nu-\mu}^{\mu+\frac12}(x)
\end{eqnarray}
where $2\mu+1$, $\nu-\mu+1\not\in- \mathbb N_0$.
One also has the following transformation for the Gegenbauer functions of the first kind  in the trigonometric context,
\begin{equation}
\hspace{0.3cm}{\sf C}_{\lambda}^{\alpha}(x)=-\frac{\sin(\pi\lambda)}{\sin(\pi(2\alpha+\lambda))}{\sf C}_{-\lambda-2\alpha}^\alpha(x),
\end{equation}
which follows from \eqref{GegCtran1}. 
The Gegenbauer functions in the trigonometric context are related to the Ferrers functions as follows.
\begin{thm}
\label{thmGegcutFer}
Let $\nu,\mu\in\CC$, $x\in\CC\setminus((-\infty,1]\cup[1,\infty))$. Then
\begin{eqnarray}
\label{ClmFerP}
&&\hspace{-6.5cm}{\sf C}_\lambda^\alpha(x)=\frac{\sqrt{\pi}\,\Gamma(\lambda+2\alpha)}{2^{\alpha-\frac12}\Gamma(\alpha)\Gamma(\lambda+1)(1-x^2)^{\frac{\alpha}{2}-\frac14}}{\sf P}_{\lambda+\alpha-\frac12}^{\frac12-\alpha}(x)\\
\label{DlmFerQ}
&&\hspace{-6.5cm}{\sf D}_\lambda^\alpha(x)=\frac{\Gamma(\lambda+2\alpha)}{\sqrt{\pi}\,2^{\alpha-\frac32}\Gamma(\alpha)\Gamma(\lambda+1)(1-x^2)^{\frac{\alpha}{2}-\frac14}}{\sf Q}_{\lambda+\alpha-\frac12}^{\frac12-\alpha}(x).
\end{eqnarray}
\label{GegFer}
\end{thm}
\begin{proof}
Start with the definitions \eqref{Ccutdef}, \eqref{Dcutdef} using  
\eqref{ClmLegP}, \eqref{DlmLegQ}. After insertion use the hyperbolic to trigonometric rules
\eqref{trighyprules}
and comparing with the definitions of the Ferrers functions \eqref{defFerP}, \eqref{defFerQ} using analytic continuation where necessary completes the proof.
\end{proof}

\noindent Conversely, the Ferrers functions are related to the Gegenbauer functions in the trigonometric context as follows.
\begin{thm}Let $\lambda,\alpha\in\CC$, $x\in\CC\setminus((-\infty,1]\cup[1,\infty))$. Then
\begin{eqnarray}
&&\hspace{-7.5cm}{\sf P}_{\nu}^\mu(x)=\frac{\Gamma(\frac12-\mu)\Gamma(\nu+\mu+1)}{\sqrt{\pi}\,2^\mu \Gamma(\nu-\mu+1)(1-x^2)^{\frac{\mu}{2}}}{\sf C}_{\nu+\mu}^{\frac12-\mu}(x),\\
&&\hspace{-7.5cm}{\sf Q}_{\nu}^\mu(x)=\frac{\sqrt{\pi}\,\Gamma(\frac12-\mu)\Gamma(\nu+\mu+1)}{2^{\mu+1} \Gamma(\nu-\mu+1)(1-x^2)^{\frac{\mu}{2}}}{\sf D}_{\nu+\mu}^{\frac12-\mu}(x),
\end{eqnarray}
\end{thm}
\begin{proof}
Inverting Theorem \ref{thmGegcutFer} while shifting variables where necessary completes the proof.
\end{proof}

\subsection{Jacobi functions of the first and second kind}

We now consider the Jacobi functions in the hyperbolic context. These are defined in a region in the complex plane given by $\CC\setminus(-\infty,1]$.
One of the most important 
classical orthogonal polynomials, the Jacobi polynomial, is defined 
as \cite[\href{http://dlmf.nist.gov/18.5.E7}{(18.5.7)}]{NIST:DLMF}
\begin{equation}
\label{Jacobipolydef}
P_n^{(\alpha,\beta)}(x):=
\frac{\Gamma(\alpha+1+n)}{n!}
\Ohyp21{-n,n+\alpha+\beta+1}{\alpha+1}{\frac{1\!-\!x}{2}}.
\end{equation}
The Legendre polynomial (the associated Legendre function of the first kind $P_\nu^\mu$ and the Ferrers function of the first kind ${\sf P}_\nu^\mu$ with $\mu=0$ and $\nu=n\in\Z$)
is given by \cite[\href{http://dlmf.nist.gov/18.7.E9}{(18.7.9)}]{NIST:DLMF}
\[
P_n(x)=C_n^{\frac12}(x)=P_n^{(0,0)}(x),
\]
where $n\in\N_0$.
From the above definition \eqref{Jacobipolydef}, it is natural to define a function extension of the Jacobi polynomials in terms of the Gauss hypergeometric function which allows for the degree $n$ to be a complex number:
\begin{thm}
\label{Firstthm}
Let
$\alpha,\beta,\gamma\in\C$ such that
$\alpha+\gamma\not\in-\N$. Then the Jacobi function
of the first kind 
$P_\gamma^{(\alpha,\beta)}:\C\setminus(-\infty,-1]\to\C$ 
can be defined by
\begin{eqnarray}
&&\hspace{-1.8cm}P_\gamma^{(\alpha,\beta)}(z):=
\frac{\Gamma(\alpha+\gamma+1)}{
\Gamma(\gamma+1)}
\Ohyp21{-\gamma,\alpha+\beta+\gamma+1}
{\alpha+1}{\frac{1-z}{2}}
\label{Jac1}\\
&&\hspace{-0.2cm}=\frac{\Gamma(\alpha+\gamma+1)}{
\Gamma(\gamma+1)}
\left(\frac{2}{z+1}\right)^\beta\Ohyp21
{-\beta-\gamma,\alpha+\gamma+1}{\alpha+1}{\frac{1-z}{2}}\label{Jac2}\\
&&\hspace{-0.2cm}=\frac{\Gamma(\alpha+\gamma+1)}
{\Gamma(\gamma+1)}
\left(\frac{z+1}{2}\right)^\gamma\Ohyp21{-\gamma,-\beta-\gamma}
{\alpha+1}{\frac{z-1}{z+1}}\label{Jac3}\\
&&\hspace{-0.2cm}=\frac{\Gamma(\alpha+\gamma+1)}
{\Gamma(\gamma+1)}
\left(\frac{2}{z+1}\right)^{\alpha+\beta+\gamma+1}\Ohyp21
{\alpha+\gamma+1,\alpha+\beta+\gamma+1}{\alpha+1}{\frac{z-1}{z+1}}.\label{Jac4}
\end{eqnarray}
\label{bigPthm}
\end{thm}
\begin{proof}
Start with \eqref{Jacobipolydef} and replace the Pochhammer
symbol by a ratio of gamma functions using $(a)_n:=\Gamma(a+n)/\Gamma(a)$, $a\not\in\-\N_0$,
and the factorial $n!=\Gamma(n+1)$ and substitute $n\mapsto\gamma\in\C$, 
$x\mapsto z$. Application of Pfaff's $(z\mapsto z/(z-1))$ and Euler's 
$(z\mapsto z)$ transformations \cite[\href{http://dlmf.nist.gov/15.8.E1}{(15.8.1)}]{NIST:DLMF} provides the 
other three representations. This completes the proof.
\end{proof}

\noindent By starting with either \eqref{Jac3} or \eqref{Jac4} and using the transformation \cite[\href{http://dlmf.nist.gov/15.8.E4}{(15.8.4)}]{NIST:DLMF}  we can obtain the asymptotic behavior for the Jacobi function of the first kind  as $|z|\to\infty$, $\alpha+\gamma\not\in-\N$,
 \begin{equation}
\label{asympPinf2}
P_\gamma^{(\alpha,\beta)}(z)\sim\frac{\Gamma(\alpha+1)\Gamma(2\gamma+\alpha+\beta+1)}{\Gamma(\gamma+1)\Gamma(\gamma+\alpha+\beta+1)}\left(\frac{z}{2}\right)^\gamma.
\end{equation}

Similarly, there are four single Gauss hypergeometric function representations 
of the Jacobi function of the second kind.
\begin{thm}\label{thmQ}
Let $\gamma,\alpha,\beta,z\in\C$ such that 
$z\in\C\setminus[-1,1]$,
$\alpha+\gamma,\beta+\gamma\notin-\N$.
Then, the Jacobi function of the second kind
has the following Gauss hypergeometric representations
\begin{eqnarray}
&&\hspace{-1.4cm}Q_\gamma^{(\alpha,\beta)}(z) :=
\frac{2^{\alpha+\beta+\gamma}\Gamma(\alpha+\gamma+1)\Gamma(\beta+\gamma+1)}
{(z-1)^{\alpha+\gamma+1}(z+1)^\beta}
\Ohyp21{\gamma+1,\alpha+\gamma+1}{\alpha+\beta+2\gamma+2}
{\frac{2}{1-z}}
\label{dJsk1}\\[0.2cm]
&&\hspace{0.3cm}=
\frac{2^{\alpha+\beta+\gamma}\Gamma(\alpha+\gamma+1)\Gamma(\beta+\gamma+1)}
{(z-1)^{\alpha+\beta+\gamma+1}}
\Ohyp21{\beta+\gamma+1,\alpha+\beta+\gamma+1}{\alpha+\beta+2\gamma+2}
{\frac{2}{1-z}}\label{dJsk4}\\[0.2cm]
&&\hspace{0.3cm}=
\frac{2^{\alpha+\beta+\gamma}\Gamma(\alpha+\gamma+1)\Gamma(\beta+\gamma+1)}
{(z-1)^{\alpha}(z+1)^{\beta+\gamma+1}}
\Ohyp21{\gamma+1,\beta+\gamma+1}{\alpha+\beta+2\gamma+2}
{\frac{2}{1+z}}\label{dJsk3}\\[0.2cm]
&&\hspace{0.3cm}=
\frac{2^{\alpha+\beta+\gamma}\Gamma(\alpha+\gamma+1)\Gamma(\beta+\gamma+1)}
{(z+1)^{\alpha+\beta+\gamma+1}}
\Ohyp21{\alpha+\gamma+1,\alpha+\beta+\gamma+1}{\alpha+\beta+2\gamma+2}
{\frac{2}{1+z}}.
\label{dJsk2}
\end{eqnarray}
\end{thm}
\begin{proof}
Start with \cite[(4.61.5)]{Szego} or \cite[(10.8.18)]{ErdelyiTIT} and
let $n\mapsto\gamma\in\C$ and $x\mapsto z$.
Application of Pfaff's $(z\mapsto z/(z-1))$ and Euler's $(z\mapsto z)$ transformations
\cite[\href{http://dlmf.nist.gov/15.8.E1}{(15.8.1)}]{NIST:DLMF} provides the other three representations.
This completes the proof.
\end{proof}

\noindent The asymptotic behavior of $Q^{(\alpha,\beta)}_\gamma(z)$ for $\lvert z\rvert\rightarrow\infty$ with 
$\arg{z}\in(-\pi,\pi)$ is evident from any of these representations with
\begin{equation}
Q_\gamma^{(\alpha,\beta)}(z) \sim \frac{2^{\alpha+\beta+\gamma}\Gamma(\alpha+\gamma+1)\Gamma(\beta+\gamma+1)}{\Gamma(\alpha+\beta+2\gamma+2)z^{\alpha+\beta+\gamma+1}}.
\label{asympQinf}
\end{equation}

\noindent The Jacobi function of the second kind has the following behavior near the singularity at $z=1$.
By using \eqref{dJsk3} and assuming $\Re\alpha>0$ we find that
if $z\sim 1+\epsilon$ then as $\epsilon\to0^{+}$ one has
\begin{equation}
Q_\gamma^{(\alpha,\beta)}(1+\epsilon)\sim\frac{2^{\alpha-1}\Gamma(\alpha)\Gamma(\beta+\gamma+1)}{\Gamma(\alpha+\beta+\gamma+1)\epsilon^{\alpha}},
\label{JacQsing1}
\end{equation}
where $\beta+\gamma\not\in-\N$.



\medskip
\noindent 
In the hyperbolic context, one has
the following transformation for a Jacobi function of the first  kind.

\begin{thm}Let $\alpha,\beta,\gamma\in\C$, $z\in\C\setminus(-\infty,-1]$, 
$\gamma-\alpha\not\in\Z$. Then
\begin{equation}
P_{-\gamma-1}^{(\alpha,\beta)}(z)=
\frac{\sin(\pi\gamma)}{\sin(\pi(\gamma-\alpha))}
\left(\frac{2}{z+1}\right)^\beta P_{\gamma-\alpha}^{(\alpha,-\beta)}(z).
\label{JacPmgmone}
\end{equation}
\end{thm}
\begin{proof}
Start with \eqref{Jac1}, let $\gamma\mapsto-\gamma-1$, then take
$\gamma\mapsto\gamma+\alpha$ and $\beta\mapsto-\beta$ and compare with 
\eqref{Jac2}, finally take $\gamma\mapsto\gamma-\alpha$ and $\beta\mapsto-\beta$. 
This completes the proof.
\end{proof}

\subsection{Jacobi functions in the trigonometric context}

The Jacobi functions in the trigonometric context  are defined in terms of the functions of complex argument by \cite[(2.3), (2.4)]{Durand78},
\begin{eqnarray}
\label{P(x)def1}
&&\hspace{-5.85cm} {\sf P}_\gamma^{(\alpha,\beta)}(x) = \frac{1}{2}\left[P_\gamma^{(\alpha,\beta)}(x+i0)+P_\gamma^{(\alpha,\beta)}(x-i0)\right] \\
\label{P(x)def2}
&& \hspace{-6cm} \hspace{1.7cm} =\frac{i}{\pi}\left[\expe^{i\pi\alpha}Q_\gamma^{(\alpha,\beta)}(x+i0)-\expe^{-i\pi \alpha}Q_\gamma^{\alpha,\beta)}(x-i 0)\right] \\
&& \hspace{-6cm} \hspace{1.7cm} = P_\gamma^{(\alpha,\beta)}(x\pm i0),
\label{Pcutdef}\\
&&\hspace{-5.9cm}  {\sf Q}_\gamma^{(\alpha,\beta)}(x) =  \frac{1}{2}\left[\expe^{i\pi\alpha}Q_\gamma^{(\alpha,\beta)}(x+i0)+\expe^{-i\pi\alpha}Q_\gamma^{(\alpha,\beta)}(x-i0)\right].
\label{Qcutdef}
\end{eqnarray}
Note that the definition of ${\sf Q}_\gamma^{(\alpha,\beta)}(x)$ differs from that in Szeg\H{o} \cite[(4.62.9)]{Szego} and \cite[(10.8.22)]{Erdelyi}, and preserves the analogy of the ${\sf P}$s and ${\sf Q}$s to trigonometric functions which the latter do not.  For the Jacobi function of the first kind in the trigonometric context, there are several Gauss hypergeometric representations, including 
\cite[\href{http://dlmf.nist.gov/14.3.E2}{(14.3.2)}]{NIST:DLMF}
\begin{eqnarray}
&&\hspace{-5.55cm}{\sf P}_\gamma^{(\alpha,\beta)}(x)=\frac{\Gamma(\alpha+\gamma+1)}{\Gamma(\gamma+1)}\Ohyp21{-\gamma,\alpha+\beta+\gamma+1}{\alpha+1}{\frac{1-x}{2}}\label{Pcutfirst}\\
&&\hspace{-4.0cm}=\frac{\Gamma(\alpha+\gamma+1)}{\Gamma(\gamma+1)}\left(\frac{1+x}{2}\right)^\gamma\Ohyp21{-\gamma,-\beta-\gamma}{\alpha+1}{\frac{x-1}{x+1}}.
\end{eqnarray}

\noindent This leads to the special value 
\begin{equation}
{\sf P}_\gamma^{(\alpha,\beta)}(1)=\frac{\Gamma(\alpha+\gamma+1)}{\Gamma(\alpha+1)\Gamma(\gamma+1)}.
\end{equation}

\begin{thm}
\label{Qcutthm}
Let $\gamma,\alpha,\beta\in\C$ such that 
$\alpha, \beta\not\in\Z$, $\alpha+\gamma,\beta+\gamma\not\in-\N$.
Then, the Jacobi function of the second kind
in the trigonometric context 
${\sf Q}_\gamma^{(\alpha,\beta)}:\C\setminus((-\infty,1]\cup[1,\infty))\to\C$ is given by
\begin{eqnarray}
&&\hspace{-0.9cm}
{\sf Q}_\gamma^{(\alpha,\beta)}(x)=
\frac{\pi}{2\sin(\pi\beta)}
\Biggl(
\cos(\pi(\beta\!+\!\gamma))\frac{\Gamma(\beta+\gamma+1)}{\Gamma(\gamma+1)}
\Ohyp21{-\gamma,\alpha+\beta+\gamma+1}{1+\beta}{\frac{1\!+\!x}{2}}\nonumber\\[-0.0cm]
&&\hspace{0.1cm}
+\cos(\pi(\gamma\!+\!1))\frac{\Gamma(\alpha+\gamma+1)}
{\Gamma(\alpha+\beta+\gamma+1)}\left(\frac{2}{1-x}\right)^{\!\!\alpha} \!\!
\left(\frac{2}{1+x}\right)^{\!\!\beta}\hspace{-2.2mm}
\Ohyp21{\gamma+1,-\alpha\!-\!\beta\!-\!\gamma}{1-\beta}{\frac{1\!+\!x}{2}}
\hspace{-1.5mm} \Biggr)
\label{Qcut3}\\
&&\hspace{0.0cm}
=\frac{\pi}{2\sin(\pi\alpha)}
\Biggl(
\cos(\pi(\alpha\!+\!1))\frac{\Gamma(\alpha+\gamma+1)}{\Gamma(\gamma+1)}\left(\frac{1\!+\!x}{2}\right)^\gamma\Ohyp21{-\gamma,-\beta-\gamma}{1+\alpha}{\frac{x\!-\!1}{x\!+\!1}}\nonumber\\[-0.0cm]
&&\hspace{0.9cm}
+\frac{\Gamma(\beta+\gamma+1)}{\Gamma(\alpha+\beta+\gamma+1)}\left(\frac{2}{1-x}\right)^\alpha\left(\frac{2}{1+x}\right)^{\beta+\gamma+1}
\Ohyp21{\gamma+1,\beta+\gamma+1}{1-\alpha}{\frac{x\!-\!1}{x\!+\!1}}
\Biggr).\label{Qcut6}
\end{eqnarray}
\end{thm}
\begin{proof}
These Gauss hypergeometric representations follow 
by applying the definition \eqref{Qcutdef} to 
Theorem \ref{thmQ} which provides the Gauss hypergeometric 
representations for the Jacobi function of the second kind. 
The application of \eqref{Qcutdef} requires the arguments of the specific 
Gauss hypergeometric functions just above and below the ray $(1,\infty)$
over which the Gauss hypergeometric function is discontinuous.
The values of the Gauss hypergeometric function above and below
this ray may then be transformed to a region where the Gauss hypergeometric
function is continuous in a complex neighborhood of the argument of the
Gauss hypergeometric function by utilizing the transformations 
which one can find in \cite[Appendix B]{Cohletal2021}.
These transformations which map from Gauss hypergeometric functions with argument $x\pm i0$ to sums of Gauss hypergeometric functions with arguments given by $1/x$, $1-x$, $1-x^{-1}$ and
$(1-x)^{-1}$. Using this method, eight Gauss hypergeometric function representations
of the Jacobi function of the second kind in the trigonometric context can be obtained. These are obtained by starting with
\eqref{dJsk1}-\eqref{dJsk2}, applying the transformation 
\cite[Theorem B.1]{Cohletal2021} $z\mapsto z^{-1}$ and by either utilizing 
the Euler $(z\mapsto z)$ or Pfaff $(z\mapsto z/(z-1))$ transformations
\cite[\href{http://dlmf.nist.gov/15.8.E1}{(15.8.1)}]{NIST:DLMF} as needed. 
There are certainly more Gauss hypergeometric representations which can be obtained for the Jacobi function of the second kind in the trigonometric context by applying \cite[Theorems B.2--B.4]{Cohletal2021}, but the derivation of these representation must be left to a later publication. In Theorem \ref{Qcutthm} only two of these representations are given and are obtained. The first is obtained from \eqref{dJsk4} and the second is from \eqref{dJsk3} utilizing the Euler and Pfaff transformations respectively. This completes the proof.
\end{proof}

\noindent By starting with \eqref{Jac2} and assuming $\Re\beta>0$ we find that
if $z\sim-1+\epsilon$, $|\epsilon|\ll 1$, then as $\epsilon\to0^{+}$ one has  
\begin{equation}
{\sf P}_\gamma^{(\alpha,\beta)}(-1+\epsilon)\sim-\frac{2^\beta\sin(\pi\gamma)\Gamma(\beta)\Gamma(\alpha+\gamma+1)}{\pi\Gamma(\alpha+\beta+\gamma+1)\epsilon^\beta},
\label{asympPcutn1}
\end{equation}
where $\beta,\alpha+\gamma+1\not\in-\N_0$, $\gamma\not\in\Z$.


\medskip
\noindent One has the following 
transformation for the Jacobi function of the first kind in the trigonometric context.
\begin{thm}Let $\alpha,\beta,\gamma\in\C$, $x\in\C\setminus(-\infty,-1]$, 
$\gamma-\alpha\not\in\Z$. Then
\begin{equation}
{\sf P}_{-\gamma-1}^{(\alpha,\beta)}(x)=
\frac{\sin(\pi\gamma)}{\sin(\pi(\gamma-\alpha))}
\left(\frac{2}{1+x}\right)^\beta {\sf P}_{\gamma-\alpha}^{(\alpha,-\beta)}(x).
\label{JacPmgmonet}
\end{equation}
\end{thm}
\begin{proof}
Start with \eqref{Jac1}, let $\gamma\mapsto-\gamma-1$, then take
$\gamma\mapsto\gamma+\alpha$ and $\beta\mapsto-\beta$ and compare with 
\eqref{Jac2}, finally take $\gamma\mapsto\gamma-\alpha$ and $\beta\mapsto-\beta$. 
This completes the proof.
\end{proof}

\noindent We also have a transformation between the Jacobi function of the second kind in the hyperbolic context with the Jacobi function of the first kind in the trigonometric context.

\begin{thm}Let $\gamma,\alpha,\beta\in\CCast$ such that $z\in\C\setminus[-1,1]$,
$\beta+\gamma\not\in-\N$,
$\alpha+\beta+\gamma\not\in\N_0$. Then
\begin{eqnarray}
&&\hspace{-2.8cm}Q_{\gamma}^{(\alpha,\beta)}(z)=\frac{2^{\alpha+\beta+\gamma}\Gamma(\alpha\!+\!\gamma\!+\!1)\Gamma(\beta\!+\!\gamma\!+\!1)\Gamma(-\alpha\!-\!\beta\!-\!\gamma)}
{\Gamma(\gamma\!+\!1)
(z\!+\!1)^{\alpha+\beta+\gamma+1}}
{\sf P}_{-\alpha-\beta-\gamma-1}^{(\alpha+\beta+2\gamma+1,\alpha)}\left(\frac{z\!-\!3}{z\!+\!1}\right).
\label{PQtrans3}
\end{eqnarray}
\end{thm}
\begin{proof}
Starting with \eqref{Pcutfirst}, replacing $\{\gamma,\alpha,\beta\}\mapsto\{-\alpha-\beta-\gamma-1,\alpha+\beta+2\gamma+1,\alpha\}$ and comparing with \eqref{dJsk2} completes the proof.
\end{proof}

\noindent Furthermore, we also have a transformation between the Jacobi function of the first kind in the trigonometric context with the Jacobi function of the first kind in the hyperbolic context.

\begin{thm}
Let $\gamma,\alpha,\beta\in\C$. 
Then
\begin{eqnarray}
&&\hspace{-8.0cm}{\sf P}_\gamma^{(\alpha,\beta)}(x)=
\left(\frac{x+1}{2}\right)^\gamma 
P_\gamma^{(\alpha,-\alpha-\beta-2\gamma-1)}\left(\frac{3-x}{x+1}\right),
\label{Pgtran3}
\end{eqnarray}
where \eqref{Pgtran3} is valid 
for $x\in\mathbb C\setminus(-\infty,-1]$.
\end{thm}
\begin{proof}
Starting with \eqref{Jac1}, let $z\mapsto (3-z)/(z+1)$, 
$\beta\mapsto-\alpha-\beta-2\gamma-1$. Then,
comparing with \eqref{Jac3} produces \eqref{Pgtran3}. This completes the proof.
\end{proof}

\subsection{Specializations of Jacobi functions}

We next present some identities which involve 
symmetric and antisymmetric Jacobi functions of
the first kind.
The relation between the symmetric Jacobi function
of the first kind and the Gegenbauer function
of the first kind for $z\in\C\setminus(-\infty,-1]$
is given by
\begin{equation}
\label{JacGegc}
P_\gamma^{(\alpha,\alpha)}(z)=
\frac{\Gamma(2\alpha+1)\Gamma(\alpha+\gamma+1)}
{\Gamma(\alpha+1)\Gamma(2\alpha+\gamma+1)}
C_\gamma^{\alpha+\frac12}(z),
\end{equation}
which is valid for Jacobi functions of the first kind in the hyperbolic context and the trigonometric context.
This follows by starting with \eqref{Jac1} and then comparing it to the Gauss hypergeometric representation of the Gegenbauer function of the first kind on the right-hand side using \eqref{Gegenbauerfuncdef}. Similarly, there is
\begin{equation}
P_\gamma^{(\alpha,\alpha)}(z)=
\frac{2^\alpha\Gamma(\alpha+\gamma+1)}{\Gamma(\gamma+1)(z^2-1)^{\frac12\alpha}} 
P_{\alpha+\gamma}^{-\alpha}(z),
\end{equation}
where $z\in\C\setminus(-\infty,1]$.
We also have
\begin{equation}
\label{relJacPFerP}
{\sf P}_\gamma^{(\alpha,\alpha)}(x)=
\frac{2^\alpha\Gamma(\alpha+\gamma+1)}{\Gamma(\gamma+1)
(1-x^2)^{\frac12\alpha}}{\sf P}_{\alpha+\gamma}^{-\alpha}(x),
\end{equation}
where $x\in\C\setminus((-\infty,-1]\cup[1,\infty))$.

\begin{rem}The relation between the antisymmetric Jacobi function of the first kind and
the Ferrers and Gegenbauer functions of the
first kinds is
\begin{eqnarray}
&&\hspace{-6.9cm}P_\gamma^{(\alpha,-\alpha)}(x)
=\frac{\Gamma(\alpha+\gamma+1)}{\Gamma(\gamma+1)}
\left(\frac{1+x}{1-x}\right)^{\frac12\alpha}
{\sf P}_\gamma^{-\alpha}(x)\nonumber\\
&&\hspace{-5.0cm}=
\frac{\Gamma(2\alpha+1)\Gamma(\gamma-\alpha+1)}
{2^\alpha\Gamma(\gamma+1)\Gamma(\alpha+1)}
(1+x)^{\alpha}
C_{\gamma-\alpha}^{\alpha+\frac12}(x)
,
\end{eqnarray}
where $x\in\C\setminus((-\infty,-1]\cup[1,\infty))$ and the relation between 
the antisymmetric Jacobi function
of the first kind 
and the associated Legendre and
Gegenbauer function 
of the first kinds is
\begin{eqnarray}
&&\hspace{-6.7cm}P_\gamma^{(\alpha,-\alpha)}(z)=\frac{\Gamma(\alpha+\gamma+1)}
{\Gamma(\gamma+1)}\left(\frac{z+1}{z-1}\right)^{\frac12\alpha}P_\gamma^{-\alpha}(z)
\nonumber\\ &&\hspace{-4.825cm}=
\frac{\Gamma(2\alpha+1)
\Gamma(\gamma-\alpha+1)}{2^\alpha\Gamma(\gamma+1)\Gamma(\alpha+1)}
(z+1)^{\alpha}
C_{\gamma-\alpha}^{\alpha+\frac12}(z),
\end{eqnarray}
where $z\in\C\setminus(-\infty,1]$.
These are obtained by comparing \eqref{Jac1} with \eqref{FerrersPdefnGauss2F1} and
\eqref{associatedLegendrefunctionP}.
\end{rem}

\begin{rem}One has the following quadratic transformations for the symmetric Jacobi functions of the first kind 
which can be found in \cite[Theorem 4.1]{Szego}. Let $z\in\CC\setminus(-\infty,1]$, $\gamma,\alpha\in\CC$,  $\alpha+\gamma\not\in-\N$. Then
\begin{eqnarray}
&&\hspace{-6.5cm}P_{2\gamma}^{(\alpha,\alpha)}(z)=\frac{\sqrt{\pi}\,\Gamma(\alpha+2\gamma+1)}{2^{2\gamma}\Gamma(\gamma+\frac12)\Gamma(\alpha+\gamma+1)}P_\gamma^{(\alpha,-\frac12)}(2z^2-1),
\label{P2g}
\end{eqnarray}
where $\alpha+2\gamma\not\in-\N$,
$\gamma\not\in-\N+\frac12$, and
\begin{eqnarray}
&&\hspace{-6.5cm}P_{2\gamma+1}^{(\alpha,\alpha)}(z)=\frac{\sqrt{\pi}\,\Gamma(\alpha+2\gamma+2)z}{2^{2\gamma+1}\Gamma(\gamma+\frac32)\Gamma(\alpha+\gamma+1)}P_\gamma^{(\alpha,\frac12)}(2z^2-1),
\label{P2gp}
\end{eqnarray}
where 
$\alpha+2\gamma+1\not\in-\N$,
$\gamma\not\in-\N-\frac12$.
The restrictions on the parameters 
come directly by applying the restrictions on the parameters in Theorem \ref{bigPthm} to the Jacobi functions of the first kind on both sides of the relations.
\end{rem}

\noindent Note that these results imply
\begin{eqnarray}
&&\hspace{-7.5cm}P_\gamma^{(\alpha,-\frac12)}(2z^2-1)=\frac{\Gamma(\gamma+\frac12)\Gamma(\alpha+\frac12)}{\sqrt{\pi}\,\Gamma(\gamma+\alpha+\frac12)}C_{2\gamma}^{\alpha+\frac12}(z),
\label{rJmhGeg}\\
&&\hspace{-7.5cm}P_\gamma^{(\alpha,\frac12)}(2z^2-1)=\frac{\Gamma(\gamma+\frac32)\Gamma(\alpha+\frac12)}{\sqrt{\pi}\,\Gamma(\alpha+\gamma+\frac32)z}C_{2\gamma+1}^{\alpha+\frac12}(z).\label{rJhGeg}
\end{eqnarray}
which follows from \eqref{P2g}, \eqref{P2gp}. Below we present some identities which involve 
symmetric and antisymmetric Jacobi functions of
the second kind.
\begin{thm}
Two equivalent relations between the symmetric and antisymmetric Jacobi function of the
second kind and the associated Legendre function of the second kind are given by
\begin{eqnarray}
&&\hspace{-8.0cm} Q_{\gamma}^{(\alpha,\alpha)}(z)
=\frac
{2^{\alpha}\expe^{i\pi\alpha}
\Gamma(\alpha+\gamma+1)}
{\Gamma(\gamma+1)(z^2-1)^{\frac12\alpha}}
Q_{\alpha+\gamma}^{-\alpha}(z),
\label{JacQLeg}\\
&& \hspace*{-6.4cm} =\frac{2^{\alpha}\expe^{-i\pi\alpha}
\Gamma(\alpha+\gamma+1)}
{\Gamma(2\alpha+\gamma+1)(z^2-1)^{\frac12\alpha}}
Q_{\alpha+\gamma}^{\alpha}(z),
\label{JacQLegb}
\end{eqnarray}
\cite[(3.3.3.2)]{Erdelyi}, where $\alpha+\gamma\not\in-\N$. The 
relation between the antisymmetric Jacobi function of the second kind and the associated Legendre function of the second kind is given by
\begin{eqnarray}
&& \hspace{-6.7cm}Q_{\gamma}^{(\alpha,-\alpha)}(z)
=\frac
{\expe^{-i\pi\alpha}
\Gamma(\gamma-\alpha+1)}
{\Gamma(\gamma+1)}
\left(\frac{z+1}{z-1}\right)^{\frac12\alpha}
Q_{\gamma}^{\alpha}(z),
\label{JacasQLeg}\\
&&\hspace{-6.7cm}Q_\gamma^{(-\alpha,\alpha)}(z)=\frac{\expe^{-i\pi\alpha}\Gamma(\gamma-\alpha+1)}{\Gamma(\gamma+1)}\left(\frac{z-1}{z+1}\right)^{\frac12\alpha}Q_\gamma^\alpha(z),
\end{eqnarray}
where $\gamma-\alpha\not\in-\N$.
\end{thm}
\begin{proof}
By comparing \eqref{dJsk1} and \eqref{dJsk2} with \eqref{Qdefntwodivide1mz} 
and by using the Legendre duplication formula for the gamma function \cite[\href{http://dlmf.nist.gov/5.5.E5}{(5.5.5)}]{NIST:DLMF} 
one can obtain all these formulas in a straightforward way.
\end{proof}
See \cite[Section 3, (A.14)]{Cohl12pow} for an interesting
application of the symmetric relation for associated Legendre functions of the second kind.

\begin{thm}
Let $\alpha,\gamma\in\C$, $z\in\C\setminus[-1,1]$, 
$\alpha+\gamma\not\in-\N$, $\Re\gamma,\Re\beta>-1$.
Then the relations between the symmetric and antisymmetric Jacobi
functions of the second kind to the Gegenbauer function
of the second kind is given by
\begin{eqnarray}
&&\hspace{-5.2cm}Q_\gamma^{(\alpha,\alpha)}(z)=
\expe^{-i\pi(\alpha+\frac12)}
\frac{\pi\,\Gamma(2\alpha+1)\Gamma(\alpha+\gamma+1)}
{\Gamma(\alpha+1)\Gamma(2\alpha+\gamma+1)}D_\gamma^{\alpha+\frac12}(z),
\label{GegDsym}
\end{eqnarray}
where  
$\alpha\in\C\setminus\{-\frac12,-\frac32,-\frac52,\ldots\}$
and
\begin{eqnarray}
&&\hspace{-3.4cm}Q_\gamma^{(\alpha,-\alpha)}(z)=\expe^{i\pi(\alpha-\frac12)}2^{2\gamma-\alpha+1}
\frac{\Gamma(\alpha+\gamma+1)\Gamma(\frac12-\alpha)\Gamma(\gamma+\frac32)}
{\Gamma(2\gamma+2)(z-1)^\alpha}D_{\alpha+\gamma}^{\frac12-\alpha}(z),
\label{GegDasym}
\end{eqnarray}
where $\alpha\in\C\setminus\{\frac12,\frac32,\frac52,\ldots\}$, 
$\gamma\in\C\setminus\{-\frac32,-\frac52,-\frac72,\ldots\}$.
\end{thm}
\begin{proof}
Start with the definition of the Jacobi function
of the second kind \eqref{dJsk1} and take
$\beta=\alpha$. Then comparing \eqref{GegDhyper} using
Euler's $(z\mapsto z)$ transformation
\cite[\href{http://dlmf.nist.gov/15.8.E1}{(15.8.1)}]{NIST:DLMF} produces 
\eqref{GegDsym}.
In order to produce \eqref{GegDasym}, start 
with \eqref{dJsk1} and take
$\beta=-\alpha$. Then compare \eqref{GegDhyper} using
Euler's $(z\mapsto z)$ transformation
\cite[\href{http://dlmf.nist.gov/15.8.E1}{(15.8.1)}]{NIST:DLMF}. This completes the
proof.
\end{proof}

\noindent One has the following quadratic transformations for symmetric Jacobi functions of the second kind.
\begin{thm}Let $z\in\CC\setminus[-1,1]$, $\gamma,\alpha\in\CC$,  $\alpha+\gamma\not\in-\N$. Then
\begin{eqnarray}
&&\hspace{-6.5cm}Q_{2\gamma}^{(\alpha,\alpha)}(z)=\frac{\sqrt{\pi}\,\Gamma(\alpha+2\gamma+1)}{2^{2\gamma}\Gamma(\gamma+\frac12)\Gamma(\alpha+\gamma+1)}Q_\gamma^{(\alpha,-\frac12)}(2z^2-1),
\label{Q2g}
\end{eqnarray}
where $\alpha+2\gamma\not\in-\N$,
$\gamma\not\in-\N+\frac12$, and
\begin{eqnarray}
&&\hspace{-6.5cm}Q_{2\gamma+1}^{(\alpha,\alpha)}(z)=\frac{\sqrt{\pi}\,\Gamma(\alpha+2\gamma+2)z}{2^{2\gamma+1}\Gamma(\gamma+\frac32)\Gamma(\alpha+\gamma+1)}Q_\gamma^{(\alpha,\frac12)}(2z^2-1),
\label{Q2gp}
\end{eqnarray}
where 
$\alpha+2\gamma+1\not\in-\N$,
$\gamma\not\in-\N-\frac12$.
\end{thm}

\begin{proof}
Starting with the left-hand sides of 
\eqref{Q2g}, \eqref{Q2gp} using the Gauss hypergeometric definition \eqref{dJsk1},
the ${}_2F_1$'s become of a form 
where $c=2b$. Then for both equations we use the quadratic
transformation of the Gauss hypergeometric function \cite[\href{http://dlmf.nist.gov/15.8.E14}{(15.8.14)}]{NIST:DLMF}. This transforms the ${}_2F_1$ to a form which is recognizable with the right-hand sides through \eqref{dJsk2}, \eqref{dJsk1}, respectively. This completes the proof.
The restrictions on the parameters 
come directly by applying the restrictions on the parameters in Theorem \ref{thmQ} to the Jacobi functions of the second kind on both sides of the relations.
\end{proof}

There is also an interesting alternative additional 
quadratic transformation for the Jacobi function 
of the second kind with $\alpha=\pm \frac12$. Note that 
there does not seem to be a corresponding formula
for the Jacobi function of the first kind since in 
this case the functions which would appear 
on the left-hand side would be a sum of two Gauss hypergeometric
functions.

\begin{thm}
\label{altquadQalpha}
Let $z\in\C$ such that $|z|<1$, $\beta,\gamma\in\C$ such that 
$\beta+\gamma+\frac12\not\in-\N_0$. Then
\begin{eqnarray}
&&\hspace{-3cm}C_{2\gamma+1}^\beta(z)=\frac{2^{2\gamma+2}\Gamma(\beta+\gamma+\frac12)}{\Gamma(-\gamma-\frac12)\Gamma(2\gamma+2)\Gamma(\beta)(1-z^2)^{\beta+\gamma+\frac12}}
Q_{-\gamma-1}^{(-\frac12,\beta+2\gamma+1)}\left(\frac{1+z^2}{1-z^2}\right),\\
&&\hspace{-3cm}C_{2\gamma}^\beta(z)=\frac{2^{2\gamma+1}\Gamma(\beta+\gamma+\frac12)z}{\Gamma(-\gamma+\frac12)\Gamma(2\gamma+1)\Gamma(\beta)(1-z^2)^{\beta+\gamma+\frac12}}
Q_{-\gamma-1}^{(\frac12,\beta+2\gamma)}\left(\frac{1+z^2}{1-z^2}\right).
\end{eqnarray}
\end{thm}
\begin{proof}
These results are easily verified by starting with \eqref{dJsk1}, \eqref{dJsk3}, substituting the related values in the Jacobi function of the second kind and comparing with associated
Legendre functions of the first kind with argument $\sqrt{(z-1)/(z+1)}$. Then utilizing a quadratic transformation 
of the Gauss hypergeometric function 
completes the proof.
\end{proof}

\begin{rem}
Note that in Theorem \ref{altquadQalpha}, if the argument
of the Jacobi function of the second kind has modulus greater than unity then the argument of the Gegenbauer function of the first kind has modulus less than unity.
\end{rem}

\noindent
\begin{cor}
Let $z,\beta,\gamma\in\C$ such that $z\in\C\setminus[-1,1]$. Then
\begin{eqnarray}
&&\hspace{-2.5cm}Q_\gamma^{(\frac12,\beta)}(z)=
\frac{2^{\beta+3\gamma+\frac52}\Gamma(-2\gamma-1)\Gamma(\gamma+\frac32)\Gamma(\beta+2\gamma+2)}
{\Gamma(\beta+\gamma+\frac32)(z-1)^\frac12(z+1)^{\beta+\gamma+1}}
C_{-2\gamma-2}^{\beta+2\gamma+2}\left(\sqrt{\frac{z-1}{z+1}}\right),
\end{eqnarray}
where $-2\gamma-1,\gamma+\frac32,\beta+2\gamma+2\not\in-\N_0$, and 
\begin{eqnarray}
&&\hspace{-2.8cm}Q_\gamma^{(-\frac12,\beta)}(z)=
\frac{2^{\beta+3\gamma+\frac12}\Gamma(-2\gamma)\Gamma(\gamma+\frac12)\Gamma(\beta+2\gamma+1)}
{\Gamma(\beta+\gamma+\frac12)(z+1)^{\beta+\gamma+\frac12}}
C_{-2\gamma-1}^{\beta+2\gamma+1}\left(\sqrt{\frac{z-1}{z+1}}\right),
\end{eqnarray}
where $-2\gamma,\gamma+\frac12,\beta+2\gamma+1\not\in-\N_0$.
\end{cor}
\begin{proof}
Inverting Theorem \ref{altquadQalpha} completes the proof.
\end{proof}

\noindent 
Note that the above results imply the following corollary.
\begin{cor}
Let $z,\beta,\gamma\in\C$ such that $z\in\C\setminus[-1,1]$, 
$\gamma+\frac32,\beta+\gamma+1\not\in-\N_0$. Then
\begin{eqnarray}
&&\hspace{-5.7cm}Q_\gamma^{(\frac12,\beta)}(z)=\frac{\Gamma(\gamma+\frac32)\Gamma(\beta+\gamma+1)}{\Gamma(\gamma+1)\Gamma(\beta+\gamma+\frac32)}
\left(\frac{2}{z-1}\right)^\frac12
Q_{\gamma+\frac12}^{(-\frac12,\beta)}(z).
\end{eqnarray}
\end{cor}
\begin{proof}
Equating the two relations in Theorem \ref{altquadQalpha}
completes the proof.
\end{proof}

\noindent The symmetric and anti-symmetric relations for Jacobi functions in the trigonometric context are given as follows.

\begin{thm}
Let $x\in\C\setminus((-\infty,-1]\cup[1,\infty))$. Then
the relation between the symmetric and antisymmetric
Jacobi functions of the second kind in the trigonometric context and
the Ferrers function of the second kind are given by
\begin{eqnarray}
&&\hspace{-5.2cm}
\label{Qcutsym}{\sf Q}_\gamma^{(\alpha,\alpha)}(x)
=\frac{2^\alpha\Gamma(\alpha+\gamma+1)}{\Gamma(\gamma+1)
(1-x^2)^{\frac12\alpha}}{\sf Q}_{\gamma+\alpha}^{-\alpha}(x),
\label{FerQsym}\\
&&\hspace{-5.2cm}{\sf Q}_\gamma^{(\alpha,-\alpha)}(x)
=\frac{\Gamma(\alpha+\gamma+1)}{\Gamma(\gamma+1)}
\left(\frac{1+x}{1-x}\right)^{\frac12\alpha}
{\sf Q}_\gamma^{-\alpha}(x),
\label{FerQasym}
\end{eqnarray}
where $\alpha+\gamma\not\in-\N$,
\begin{eqnarray}
&&\hspace{-5.2cm}{\sf Q}_\gamma^{(-\alpha,\alpha)}(x)
=\frac{\Gamma(\gamma-\alpha+1)}{\Gamma(\gamma+1)}
\left(\frac{1-x}{1+x}\right)^{\frac12\alpha}
{\sf Q}_\gamma^{\alpha}(x),
\label{lastQcut}
\end{eqnarray}
where $\gamma-\alpha\not\in-\N$.
\end{thm}

\begin{proof}
The formula \eqref{FerQsym} is obtained by starting 
with \eqref{Qcut6}, taking $\beta=\alpha$, applying 
and  then comparing with \eqref{FerrersQ}. 
The formula \eqref{FerQasym} is obtained by comparing \eqref{Qcut3} 
with $\beta=-\alpha$ and cf.~\cite[Theorem 3.2]{Cohletal2021}
\begin{eqnarray}
&&\hspace{-1.0cm} {\sf Q}_\nu^\mu (x) = \frac{\pi}{2\sin(\pi\mu)}
\Bigg[
\cos(\pi(\nu + \mu))
\frac{\Gamma(\nu+\mu+1)}{
\Gamma(\nu-\mu+1)}
{\left( \frac{1+x}{1-x} \right)}^{{\frac12\mu}}
\Ohyp21{-\nu, \nu+1}{1+\mu}{\frac{1\!+\!x}{2}}
\nonumber
\\&&\hspace{5.1cm}
{}-
\cos(\pi\nu)
{\left( \frac{1-x}{1+x} \right)}^{{\frac12\mu}}
\Ohyp21{-\nu, \nu+1}{1 - \mu}{\frac{1\!+\!x}{2}} 
\Bigg]
,
\label{FerrersQ}
\end{eqnarray}
where $\nu \in \C$, $\mu \in \C \setminus \Z$, such
that $\nu + \mu \notin -\N$.
The relation \eqref{lastQcut} follows simply by taking
$\alpha\mapsto-\alpha$ in \eqref{FerQasym}.
This completes the proof.
\end{proof}


\section{Integral representations of products of Jacobi functions\label{sec:double_int}}

An important double integral representation for the product of two Jacobi functions of the second kind was given in \cite[(2.5)]{Durand78} without proof and is reproduced below in Theorem \ref{DurandQQt}. Originally it was given in terms of the alternate variables $t_i$ with $z_i =\cosh{2t_i}=2y_i^2-1$.
Let
\begin{equation}
Z:=Z(y_1,y_2,u,w):=2y_1^2y_2^2-1+4y_1y_2\sqrt{(y_1^2-1)(y_2^2-1)}uw+2(y_1^2-1)(y_2^2-1)w^2,
\end{equation}
where $y_1,y_2,u,w\in\CC\setminus(-\infty,1]$ and let $P_{\lambda}^{(\alpha,\beta)}$ and $C_\mu^{\alpha}$ be Jacobi \cite[\href{http://dlmf.nist.gov/18.5.E7}{(18.5.7)}]{NIST:DLMF} and ultraspherical (Gegenbauer) \cite[\href{http://dlmf.nist.gov/18.7.E1}{(18.7.1)}]{NIST:DLMF} functions of the first kind respectively. One has the following double integral representation of a product of two Jacobi functions of the second kind.
\begin{thm}{Durand (1978) \cite[(2.5)]{Durand78}.}
\label{DurandQQt}
Let $\lambda,\mu,\gamma,\alpha,\beta\in\CCast$, $y_1,y_2\in\CC\setminus(-\infty,1]$. Then
\begin{eqnarray}
&&\hspace{-0.2cm}Q_{\gamma-\lambda-\mu}^{(\alpha+2\lambda+\mu,\beta+\mu)}(2y_1^2-1)
Q_{\gamma-\lambda-\mu}^{(\alpha+2\lambda+\mu,\beta+\mu)}(2y_2^2-1)\nonumber\\
&&=\frac{2^{2\beta}\Gamma(\beta)\Gamma(\alpha+\gamma+1+\lambda)\Gamma(\beta+\gamma+1-\lambda)\Gamma(\alpha+\beta+\gamma+1)\Gamma(\mu+1)\Gamma(\lambda+1)y_1^{-\mu}y_2^{-\mu}}
{\Gamma(2\beta+\mu)\Gamma(\gamma+1-\lambda-\mu)\Gamma(\alpha-\beta+\lambda)\Gamma(\beta+\gamma+1)\Gamma(\alpha+\beta+\gamma+1+\lambda+\mu)(y_1^2-1)^{\lambda+\frac12\mu}(y_2^2-1)^{\lambda+\frac12\mu}}\nonumber\\
&&\hspace{0.5cm}\times \int_1^\infty \dd u
\int_1^\infty \dd w\,
Q_\gamma^{(\alpha,\beta)}(Z)P_\lambda^{(\alpha-\beta-1,\beta+\mu)}(2w^2-1)C_\mu^\beta(u)
w^{2\beta+\mu+1}(w^2-1)^{\alpha-\beta-1}(u^2-1)^{\beta-\frac12},
\label{DurandQQ}
\end{eqnarray}

\end{thm}
\begin{proof}
The proof of the product formula \eqref{DurandQQ} follows the proof of the integral representation for a product of two Gegenbauer functions of the second kind which was given in \cite{Durand75}. The proof is based on a contour integral approach and is sketched in Section \ref{intro}. 
\end{proof}
\noindent 

\begin{rem}
By replacing $2y_1^2-1,2y_2^2-1\mapsto x,y$ in \eqref{DurandQQ}, with $x,y\in\CC\setminus(-\infty,1]$, one can derive a useful equivalent form of the product formula, namely
\begin{eqnarray}
&&\hspace{-0.2cm}((x-1)(y-1))^{\lambda+\frac12\mu}((x+1)(y+1))^{\frac12\mu}\,Q_{\gamma-\lambda-\mu}^{(\alpha+2\lambda+\mu,\beta+\mu)}(x)
Q_{\gamma-\lambda-\mu}^{(\alpha+2\lambda+\mu,\beta+\mu)}(y)\nonumber\\
&&=\frac{2^{2\beta+2\lambda+2\mu}\Gamma(\beta)\Gamma(\lambda+1)\Gamma(\mu+1)\Gamma(\alpha+\gamma+\lambda+1)\Gamma(\beta+\gamma-\lambda+1)\Gamma(\alpha+\beta+\gamma+1)}
{\Gamma(\beta+\gamma+1)\Gamma(2\beta+\mu)\Gamma(\alpha-\beta+\lambda)\Gamma(\gamma-\lambda-\mu+1)\Gamma(\alpha+\beta+\gamma+\lambda+\mu+1)}\nonumber\\
&&\hspace{0.5cm}\times \int_1^\infty 
\int_1^\infty 
Q_\gamma^{(\alpha,\beta)}(W)P_\lambda^{(\alpha-\beta-1,\beta+\mu)}(2w^2-1)C_\mu^\beta(u)
w^{2\beta+\mu+1}(w^2-1)^{\alpha-\beta-1}(u^2-1)^{\beta-\frac12} \,\dd w\,\dd u,
\label{DurandQQe}
\end{eqnarray}
where
\begin{equation}
W:=W(x,y,u,w):=xy+\sqrt{x^2-1}\sqrt{y^2-1}uw+\tfrac12(x-1)(y-1)(w^2-1).
\end{equation}
\end{rem}

\medskip
\noindent Durand  put the expression in \eqref{DurandQQ} into kernel form to obtain a result analogous to that obtained in \cite{FlenstedJensenKoornwinder73} for the product of two Jacobi $P$s. However, his result, given in \cite[(2.10)]{Durand78}, is not correct as stated because of an incorrect lower limit of integration in the remaining integral. The correct lower limit in terms of the hyperbolic angles $t_i$ used there with $z_i= \cosh{t_i}$, is $t_{3,\rm{min}}=t_1+t_2$ rather than $t_{3,\rm{min}}=0$.  Here we will present the derivation of the kernel form of the product and correct the error using the variables $z_i$ directly.

\medskip
\noindent Setting $l=m=0$ in \eqref{DurandQQ}  produces the following result:
\begin{eqnarray}
&&\hspace{-2.2cm}
Q_{\gamma}^{(\alpha,\beta)}(2y_1^2-1)
Q_{\gamma}^{(\alpha,\beta)}(2y_2^2-1)=\frac{2^{2\beta}\Gamma(\beta)\Gamma(\alpha+\gamma+1)}
{\Gamma(2\beta)\Gamma(\gamma+1)\Gamma(\alpha-\beta)}\nonumber\\
&&\hspace{0.5cm}\times \int_1^\infty \dd u
\int_1^\infty \dd w\,
Q_\gamma^{(\alpha,\beta)}(Z)
w^{2\beta+1}(w^2-1)^{\alpha-\beta-1}(u^2-1)^{\beta-\frac12}.
\label{init}
\end{eqnarray}
In order to derive the kernel form of the product formula (by starting with this result), we will need the following lemma.
\begin{lem}
\label{lemD}
Let $\alpha,\beta\in\CCast$. Then
\begin{eqnarray}
\label{lemm1}
&&\int_1^z(w^2-1)^{\beta-\frac12}(z-w)^{\alpha-\beta-1}\,\dd w=2^{\beta-\frac{1}{2}}\Gamma(\beta+\tfrac12)\Gamma(\alpha-\beta)(z-1)^{\alpha-\frac12}\Ohyp21{\frac12-\beta,\beta+\frac{1}{2}}{\alpha+\frac12}{\frac{1-z}{2}} \\
\label{lemm2}
&&\hspace{2cm}=\frac{\expe^{-i\pi\beta}2^{\beta}\Gamma(\beta+\frac12)\Gamma(\alpha-\beta)}{\sqrt{\pi}\,\Gamma(\alpha+\beta)}(z^2-1)^{\frac12(\alpha-1)}Q_{\alpha-1}^{\beta}\left(\frac{z}{\sqrt{z^2-1}}\right) \\
\label{lemm3}
&&\hspace{2cm}=-\frac{\pi\,\expe^{i\pi(\beta-\frac12)}}{\sin(\pi(\beta-\frac12))}(z^2-1)^{\frac12(\alpha+\beta-1)}D_{\alpha+\beta-1}^{\frac12-\beta}\left(\frac{z}{\sqrt{z^2-1}}\right).
\end{eqnarray}
\end{lem}
\begin{proof} Let $u=(w-1)/(z-1)$, $0\leq u\leq1$. Then
\begin{eqnarray}
\label{lem1b}
&&\int_1^z(w^2-1)^{\beta-\frac12}(z-w)^{\alpha-\beta-1}\,\dd w=2^{\beta-\frac{1}{2}}(z-1)^{\alpha-\frac{1}{2}}\int_0^1 u^{\beta-\frac{1}{2}}(1-u)^{\alpha-\beta-1}\left(1-\frac{1-z}{2}u\right)^{\beta-\frac{1}{2}} \dd u \\
\label{lem2}
&&\hspace{4cm}  = 2^{\beta-\frac{1}{2}}\Gamma(\beta+\tfrac12)\Gamma(\alpha-\beta)(z-1)^{\alpha-\frac12}\Ohyp21{\frac12-\beta,\beta+\frac{1}{2}}{\alpha+\frac12}{\frac{1-z}{2}} \\
\label{lem2b}
&&\hspace{4cm}= 2^{\beta-\frac{1}{2}}\Gamma(\beta+\tfrac{1}{2})\Gamma(\alpha-\beta)\left(z^2-1\right)^{\frac{\alpha}{2}-\frac{1}{4}}P_{\beta-\frac{1}{2}}^{-\alpha+\frac{1}{2}}(z),
\end{eqnarray}
by \cite[\href{http://dlmf.nist.gov/15.6.E1}{(15.6.1)}]{NIST:DLMF} and  \eqref{associatedLegendrefunctionP}. The Whipple transformation \cite[(3.3.1.14)]{Erdelyi}  then gives \eqref{lemm2} and the relation between associated Legendre and Gegenbauer functions in \eqref{QnmtoDla} gives \eqref{lemm3}. Alternatively,
\begin{eqnarray}
\label{lem4}
&&\int_1^z(w^2-1)^{\beta-\frac12}(z-w)^{\alpha-\beta-1}\,\dd w = 
2^{\beta-\alpha}\Gamma(\beta+\tfrac12)\Gamma(\alpha-\beta)(z^2-1)^{\alpha-\frac12}\Ohyp21{\alpha+\beta,\alpha-\beta}{\alpha+\frac12}{\frac{1-z}{2}} \\
\label{lem5}
&&\hspace{2cm}=\Gamma(\beta+\tfrac12)\Gamma(\alpha-\beta)(z-1)^{\alpha-\frac{1}{2}}(z+1)^{\beta-\frac{1}{2}}\Ohyp21{\frac12-\beta,\alpha-\beta}{\alpha+\frac{1}{2}}{\frac{z-1}{z+1}}
\end{eqnarray}
from \eqref{lem2} and \cite[\href{http://dlmf.nist.gov/15.8.E1}{(15.8.1)}]{NIST:DLMF}.
\end{proof}

\begin{rem} Equation \eqref{lem4} gives the form of the integral used in \cite[2.2 (2.10)]{Durand78} and in \cite[(4.19)]{FlenstedJensenKoornwinder73}. The form holds for $\lvert(1-z)/2\rvert<1$; the result is defined elsewhere by analytic continuation of the hypergeometric function, not considered explicitly in \cite{Durand78} but given by linking the hypergeometric function to Legendre or Gegenbauer functions as above. The results here and in \cite{Durand78} are equivalent. The error in the latter is in the lower limit of integration in the cited equation. 
\end{rem}

\noindent We next present the corrected version of Durand's kernel formula and its proof.
\begin{thm}
\label{bigQthm}
Let $\alpha,\beta,\gamma\in\CCast$, $y_1,y_2\in\CC\setminus(-\infty,1]$, $\Re\alpha>-\frac12$,  $\Re(\beta+\gamma)>-1$, $\Re\gamma>1$. Then
\begin{eqnarray}
&&\hspace{-0.0cm}Q_{\gamma}^{(\alpha,\beta)}(2y_1^2-1)
Q_{\gamma}^{(\alpha,\beta)}(2y_2^2-1)
=\frac{\expe^{-i\pi(\beta+\frac12)}2^{\alpha+\beta}\sqrt{\pi}\,\Gamma(\beta+\frac12)\Gamma(\alpha+\gamma+1)(y_1y_2)^{\alpha-\beta-1}}{\Gamma(\gamma+1)\Gamma(\alpha+\beta)(y_1^2-1)^\alpha(y_2^2-1)^\alpha}\nonumber\\
&&\hspace{1.4cm}\times \int_{y_1y_2+\sqrt{(y_1^2-1)(y_2^2-1)}}^\infty
\dd y_3\,Q_\gamma^{(\alpha,\beta)}(2y_3^2-1)y_3^{\alpha+\beta}\left(A^2-1\right)^{\frac12(\alpha-\beta-1)}
D_{\alpha-\beta-1}^{\beta+\frac12}\left(\frac{A}{\sqrt{A^2-1}}\right)\label{JacobiQQD}\\
&&\hspace{0.6cm}=\frac{\expe^{-i\pi\beta}2^{\alpha}\Gamma(\alpha+\gamma+1)(y_1y_2)^{\alpha-\beta-1}}{\Gamma(\gamma+1)\Gamma(\alpha+\beta)(y_1^2-1)^\alpha(y_2^2-1)^\alpha}\nonumber\\
&&\hspace{1.4cm}\times \int_{y_1y_2+\sqrt{(y_1^2-1)(y_2^2-1)}}^\infty
\dd y_3\,Q_\gamma^{(\alpha,\beta)}(2y_3^2-1)y_3^{\alpha+\beta}\left(A^2-1\right)^{\frac12(\alpha-1)}
Q_{\alpha-1}^{\beta}\left(\frac{A}{\sqrt{A^2-1}}\right),
\label{JacQQQ}
\end{eqnarray}
where 
\begin{equation}
A:=A(y_1,y_2,y_3):=\frac{y_1^2+y_2^2+y_3^2-1}{2\,y_1y_2y_3}.
\label{Adef}
\end{equation}

\end{thm}
\begin{proof}
Start with \eqref{init} and make the following transformation of integration variables
$(u,w)\mapsto(y_3,G)$, such that
\begin{eqnarray}
&&\hspace{-7.6cm}u:=u(y_1,y_2,y_3,G)=\frac{Gy_3-y_1y_2}{\sqrt{y_1^2y_2^2+y_3^2-2y_1y_2y_3G}},\label{tru}\\&&\hspace{-7.6cm}
w:=w(y_1,y_2,y_3,G)=\sqrt{\frac{y_1^2y_2^2+y_3^2-2y_1y_2y_3G}{(y_1^2-1)(y_2^2-1)}},\label{trG}
\end{eqnarray}
and $y_3,G\in\CC\setminus(-\infty,1]$. Corresponding to this transformation there is also the inverse transformation given by
\begin{eqnarray}
&&\hspace{-0.5cm}y_3=\sqrt{y_1^2y_2^2+2y_1y_2\sqrt{(y_1^2-1)(y_2^2-1)}uw+
({y_1^2-1})({y_2^2-1})
w^2},\\
&&\hspace{-0.5cm}G=\frac{\sqrt{y_1^2y_2^2+2y_1y_2\sqrt{(y_1^2-1)(y_2^2-1)}uw+(y_1^2-1)(y_2^2-1)w^2}}{y_1^4y_2^4-2y_1^2y_2^2(y_1^2-1)(y_2^2-1)(2u^2-1)w^2+(y_1^2-1)^2(y_2^2-1)^2w^4}\nonumber\\
&&\times(y_1^3y_2^3\!-\!y_1^2y_2^2\sqrt{(y_1^2-1)(y_2^2-1)}uw\!-\!y_1y_2(y_1^2-1)(y_2^2-1)(2u^2\!-\!1)w^2\!+\!((y_1^2-1)(y_2^2-1))^\frac32uw^3).
\end{eqnarray}
First replace the variables in  the integrand of \eqref{init}
using \eqref{tru}, \eqref{trG}. By design, this
converts $Z\mapsto 2y_3^2-1$.
Now in order to compute the limits of integration for the transformed coordinates, one must compute the following limits
\begin{eqnarray}
&&\hspace{-5.0cm}\lim_{w,u\to\infty}G= \lim_{w,u\to\infty}y_3=\infty,\nonumber\\
&&\hspace{-5.0cm}\lim_{w,u\to 1}G=1,\
\lim_{w,u\to 1}y_3=y_1y_2+\sqrt{(y_1^2-1)(y_2^2-1)}.
\end{eqnarray}
Now in order to compute the transformed area measure, one must compute the absolute value of the Jacobian determinant, namely
\begin{equation}
\hspace{0.5cm}\dd u\,\dd w=\left|\frac{\partial(u,w)}{\partial(y_3,G)}\right|\dd G\,\dd y_3=\frac{y_3^2(y_1^2-1)^{-\frac12}(y_2^2-1)^{-\frac12}}{y_1^2y_2^2+y_3^2-2y_1y_2y_3G}\dd G\,\dd y_3.
\end{equation}
To ensure that only positive values of $(w^2-1)^{\alpha-\beta-1}$ are taken we must restrict $G\in(1,A)$ where $A$ is defined using \eqref{Adef}. One uses Lemma \ref{lem2} to compute the integral over $G$.
The restriction on the parameters is obtained by performing an asymptotic analysis of the integrand about the singularity at large values of $z_3$. Simplification completes the proof.
\end{proof}

\begin{rem}
For Theorem \ref{bigQthm}, due to the restrictions on parameters, there are no solutions for $\gamma=0,1$. However, there are solutions for $\gamma\in\N$ such that $\gamma\ge 2$, namely for $\Re\alpha>-\frac12$.
\end{rem}

\begin{rem}
As Koornwinder noted in  \cite[(3)]{Koornwinder1972AI}, the variable $Z$ in the addition formula for Jacobi polynomials can be written in the trigonometric context as $Z=2\lvert \cos{\theta_1}\cos{\theta_2}+r\expe^{i\phi}\sin{\theta_1}\sin{\theta_2}\rvert^2-1$. This observation was used in \cite[p.~255]{FlenstedJensenKoornwinder73} by Flensted-Jensen and Koornwinder with the transformation of variables $\expe^{i\chi}\cosh{t_3}=\cosh{t_1}\cosh{t_2}+\sinh{t_1}\sinh{t_2}\,r\expe^{i\psi}$ from $(r,\,\psi)$ to $(t_3, \chi)$  in their analytic derivation of the Gasper's  product formula \eqref{Gasper_RR}. The transformation used by Durand in \cite[(2.9)]{Durand78},
\begin{eqnarray}
&&\hspace{-6cm}\expe^g\cosh t_3=\cosh t_1\cosh t_2+\sinh t_1 \sinh t_2 \cosh\psi \,\expe^\phi,\label{t1}\\
&&\hspace{-6cm}\expe^{-g}\cosh t_3=\cosh t_1\cosh t_2+\sinh t_1 \sinh t_2 \cosh\psi \,\expe^{-\phi},\label{t2}
\end{eqnarray}
generalizes Koornwinder's transformation to the completely hyperbolic context. When the two equations \eqref{t1}, \eqref{t2} are multiplied together, one  obtains $(\cosh{t_3})^2=(Z+1)/2$, while the ratio of the equations determines $\expe^{2g}$. After this transformation, the integral over  $g$ can be performed giving Lemma \ref{lemD} in the form in (\ref{lem4}), and the initial double integral reduced to a single integral over $t_3$. In order to simplify the computations and to avoid the use of inverse hyperbolic functions, we have adopted an algebraic formalism by making the replacements $y_1=\cosh t_1$, $y_2=\cosh t_2$, $y_3=\cosh t_3$, $u=\cosh\phi$ and $w=\cosh\psi$, $G=\cosh g$. The integration variables $\phi,\psi,t_3,g\in(0,\infty)$ and therefore  $u,w,y_3,G\in(1,\infty)$. By choosing
$\expe^{\pm \phi}=u\pm\sqrt{u^2-1}$, $\expe^{\pm g}=G\pm\sqrt{G^2-1}$, and the above replacements, the transformations  \eqref{tru}, \eqref{trG} are readily obtained.
\end{rem}

\noindent By using the transformation of a Jacobi function of the second kind in the hyperbolic context to a Jacobi function of the first kind in the trigonometric context \eqref{PQtrans3}, we can obtain a product integral representation for a product of two Jacobi functions of the first kind in the trigonometric context.

\begin{thm}
\label{thm23}
Let $\alpha,\beta,\gamma\in\CCast$, $x_1,x_2\in\CC\setminus((-\infty,-1]\cup(1,\infty))$, $\Re(\beta+\gamma)<0$, $\Re\beta>-\frac12$. Then
\begin{eqnarray}
&&\hspace{-0.5cm}{\sf P}_\gamma^{(\alpha,\beta)}(2x_1^2\!-\!1){\sf P}_\gamma^{(\alpha,\beta)}(2x_2^2\!-\!1)
=
\frac{2^{2\beta}(\alpha+2\gamma)(\alpha+2\gamma+1)\Gamma(\beta)\Gamma(\alpha\!+\!2\gamma)}
{\Gamma(\gamma+1)\Gamma(-\beta\!-\!\gamma)\Gamma(\alpha+2\beta+2\gamma+1)
(x_1x_2)^{2\beta}(1\!-\!x_1^2)^{\frac12\alpha}(1\!-\!x_2^2)^{\frac12\alpha}}
\nonumber\\
&&\hspace{1.5cm}\times\int_{\frac{x_1+x_2}{1+x_1x_2}}^1\dd x_3\,
{\sf P}_\gamma^{(\alpha,\beta)}(2x_3^2-1)\,x_3(1-x_3^2)^{\frac12(\alpha-2)}
(B^2-1)^{\beta-\frac12}
C_{\alpha+2\gamma+1}^\beta
\left(B\right),
\label{prodPPcut}
\end{eqnarray}
where
\begin{equation}
B:=B(x_1,x_2,x_3):=\frac{2-x_1^2-x_2^2-x_3^2+x_1^2x_2^2x_3^3}{2\sqrt{(1-x_1^2)(1-x_2^2)(1-x_3^2)}}.
\label{Bdef}
\end{equation}
\end{thm}
\begin{proof}
First apply the transformation \eqref{PQtrans3} to the Jacobi functions of the second kind on the left-hand side of \eqref{JacobiQQD} and also to the Jacobi function of the second kind in the integrand on the right-hand side. After applying this transformation replace the parameters $\{\gamma,\alpha,\beta\}\mapsto\{\alpha+\gamma,\beta,-\alpha-\beta-2\gamma-1\}$. Then make the argument replacements
$y_k=(1-x_k^2)^{-\frac12}$ for $k=1,2,3$, or $x_k=(y_k^2-1)^{\frac12}/y_k$. Then the arguments of the resulting Jacobi functions of the first kind become $2x_k^2-1$ for $k=1,2,3$.
As $y_3\to\infty$, then $x_3\to 1$, and the other endpoint of integration becomes $(x_1+x_2)/(1+x_1x_2)$. To obtain the double integral representation in terms of the Gegenbauer function of the first kind use the Whipple transformation \eqref{WhippleD} but instead write $z=B/\sqrt{B^2-1}$ and utilize the Szeg\H{o} transformation \cite[(1), See also Appendix A]{Cohlonpar}. This produces a Gegenbauer function of the first kind with degree given by $-\alpha-2\beta-2\gamma-1$ which can be converted to one with degree with $\alpha+2\gamma+1$ by using the transformation \eqref{GegCtran1}. In order to obtain the restriction on the parameters, perform an asymptotic analysis of the integrand about the singularity at $X_3\sim 1$. Simplification completes the proof.
\end{proof}

\begin{rem}
Given the constraints on the parameters in Theorem \ref{thm23},  for $\gamma=0$ the integral representation converges for $-\frac12<\Re\beta<0$.
However, there does not exist any solutions for the parameters with $\gamma\in\N$, so the above double integral representation for for the Jacobi function of the first kind in the trigonometric context does not provide a double integral representation for Jacobi polynomials.
\end{rem}

\noindent By starting with the the integral representation given in Theorem \ref{thm23}, we can also obtain an integral representation for a product of two Jacobi functions of the first kind in the hyperbolic context by utilizing the transformation \eqref{Pgtran3}.
\begin{thm}
\label{thm14}
Let $y_1,y_2\in\CC\setminus(-\infty,1]$, $\alpha,\beta,\gamma\in\CCast$, $\Re(\alpha+\gamma), \Re(\alpha+\beta+\gamma)>-1$, $\Re(\alpha+\beta+2\gamma)<-\frac12$. Then
\begin{eqnarray}
&&\hspace{-0.9cm}P_\gamma^{(\alpha,\beta)}(2y_1^2-1)P_\gamma^{(\alpha,\beta)}(2y_2^2-1)\nonumber\\
&&\hspace{-0.5cm}=\frac{\Gamma(\alpha+2\gamma+2)\Gamma(-\alpha-\beta-2\gamma-1)}{2^{2\alpha+2\beta+4\gamma+2}\Gamma(\gamma+1)\Gamma(\alpha+\beta+\gamma+1)\Gamma(-\alpha-2\beta-2\gamma-1)(y_1y_2)^{\alpha+2\beta+2\gamma+2}(y_1^2-1)^{\frac12\alpha}(y_2^2-1)^{\frac12\alpha}}
\nonumber\\
&&\hspace{1.0cm}\times\int_{1}^{\frac{y_1y_2+1}{y_1+y_2}}\dd y_3\,P_\gamma^{(\alpha,\beta)}(2y_3^2-1)\frac{(y_3^2-1)^{\frac12(\alpha-2)}}{y_3^{\alpha+2\gamma+1}}(C^2-1)^{-\alpha-\beta-2\gamma-\frac32}C_{\alpha+2\gamma+1}^{-(\alpha+\beta+2\gamma+1)}(C),
\end{eqnarray}
where
\begin{equation}
C:=C(y_1,y_2,y_3):=\frac{1-y_1^2y_2^2-y_1^2y_3^2-y_2^2y_3^2+2y_1^2y_2^2y_3^2}{2y_1y_2y_3\sqrt{(y_1^2-1)(y_2^2-1)(y_3^2-1)}}.
\label{Cdef}
\end{equation}
\end{thm}
\begin{proof}
First apply the transformation \eqref{Pgtran3} to the Jacobi functions of the first kind on the left-hand side of \eqref{prodPPcut} and also to the Jacobi function of the first kind in the integrand on the right-hand side. After applying this transformation replace the parameter $\beta\mapsto-\beta-2\gamma-1$. Then make the argument replacements $y_k=x_k^{-1}$ for $k=1,2,3$. Then the arguments of the resulting Jacobi functions of the first kind in the hyperbolic context become $2y_k^2-1$ for $k=1,2,3$.
As $x_3\to1$, then $y_3\to 1$. The other endpoint of integration becomes $(y_1y_2+1)/(y_1+y_2)$. In order to obtain the restriction on the parameters, perform asymptotic analysis of the integrand about that singularity at $y_3\sim 1$. Simplification completes the proof.
\end{proof}

\begin{rem}
Given the constraints on the parameters in Theorem \ref{thm14},  for $\gamma=0$ the integral representation converges for $\Re\alpha>-1$, and $\Re(-\alpha-1)<\Re\beta<\Re(-\alpha-\frac12)$.
However, there does not exist any solutions for the parameters with $\gamma\in\N$, so the above double integral representation for for the Jacobi function of the first kind in the hyperbolic context does not provide a double integral representation for Jacobi polynomials.
\end{rem}

\section{Integral representations for products of Gegenbauer and associated Legendre  functions
\label{sec_other_double_ints}}

First we can obtain a double integral representation of a product of two associated Legendre functions of the second kind.
\begin{cor}
Let $z_1,z_2\in\CC\setminus(-\infty,1]$, $\nu,\mu\in\CCast$, such that $\Re(\nu-\mu)>-1$. Then
\begin{eqnarray}
&&\hspace{-3.5cm}Q_\nu^\mu(z_1)Q_\nu^\mu(z_2)=\frac{\expe^{i\pi\mu}\sqrt{\tfrac{\pi}{2}}\,\Gamma(\nu+\mu+1)}{\Gamma(\mu+\frac12)\Gamma(\nu-\mu+1)}\frac{(z_1z_2)^{\mu-\frac12}}{(z_1^2-1)^{\frac12\mu}(z_2^2-1)^{\frac12\mu}}
\nonumber\\
&&\hspace{-0.5cm}\times
\int_{z_1z_2+\sqrt{(z_1^2-1)(z_2^2-1)}}^\infty\dd z_3\,Q_\nu^\mu(z_3)\,\frac{\left(\frac{z_1^2+z_2^2+z_3^2-1}{2z_1z_2}-z_3\right)^{\mu-\frac12}}{(z_3^2-1)^{\frac12\mu}}.
\label{LegQQint}
\end{eqnarray}
\end{cor}
\begin{proof}
Start with Theorem \ref{bigQthm}, and in particular \eqref{JacQQQ}, and set $\beta=\alpha$, then apply \eqref{JacQLegb}
to the Jacobi functions of the second kind on the left-hand side and as well the Jacobi function of the second kind on the right-hand side. Then make the transformation $y_k\mapsto\sqrt{(z_k+1)/2}$, for $k=1,2,3$. This maps the arguments of the associated Legendre functions of the second kind to the original variables $z_k$  in (\ref{Z1}). The associated Legendre function $Q_{\mu-1}^{\mu}$ can be expressed as an algebraic function using Lemma \ref{lem3}. In order to obtain the restriction on the parameters, perform an asymptotic analysis of the leading term as $z_3\to \infty$. Simplification completes the proof.
\end{proof}

\noindent We can also obtain an integral representation of a product of two Gegenbauer functions of the second kind.
\begin{cor}Let $z_1,z_2\in\CC\setminus(-\infty,1]$, $\nu,\mu\in\CCast$, such that $\Re(\lambda+2\alpha)>0$. Then
\begin{eqnarray}
&&\hspace{-2cm}D_\lambda^{\alpha}(z_1)D_\lambda^{\alpha}(z_2)=\frac{\expe^{i\pi\alpha}\Gamma(\lambda+2\alpha)(z_1z_2)^{\alpha-1}}{2^{\alpha}\Gamma(\lambda+1)[\Gamma(\alpha)]^2(z_1^2-1)^{\alpha-\frac12}(z_2^2-1)^{\alpha-\frac12}} \nonumber \\
\label{DDkernel}
&&\hspace{0.9cm}\times \int_{z_1z_2+\sqrt{(z_1^2-1)(z_2^2-1)}}^\infty \dd z_3\,D_\lambda^{\alpha}(z_3)\left(\frac{z_1^2+z_2^2+z_3^2-1}{2z_1z_2}-z_3\right)^{\alpha-1}.
\end{eqnarray}
\end{cor}
%
%
\begin{proof}
Start with Theorem \ref{bigQthm}, \eqref{JacQQQ}, with $\beta=\alpha$, make the change of variables $y_k\mapsto\sqrt{(z_k+1)/2}$, and use \eqref{GegDsym} on both sides of the equation with appropriate substitutions of the parameters $\{\gamma,\alpha\}\mapsto\{\lambda,\alpha\}$. Alternatively,
start with \eqref{LegQQint} and replace the associated Legendre functions with Gegenbauer functions of the second kind  on both sides of the equation
using \eqref{QnmtoDla} and making appropriate substitutions of the variables $\{\nu,\mu\}\mapsto\{\lambda,\alpha\}$. An asymptotic analysis of the integrand as $z_3\to\infty$ provides the restrictions on the parameters. Simplification completes the proof.
\end{proof}
\begin{rem}
The result in \eqref{DDkernel} follows directly from Durand's product formula for two $D$s  in \eqref{DDproduct} by replacing the integration variable $t$ by $z_3=z_1z_2+t\sqrt{z_1^2-1}\sqrt{z_2^2-1}$. That product formula was derived directly starting with the extension of Gegenbauer's addition formula      for $C_\lambda^\alpha(z_1z_2+t\sqrt{1-z_1^2}\sqrt{1-z_2^2})$ to the corresponding formula for a $D$, and is independent of the double integral expressions considered above.
\end{rem}

\noindent We also have an integral representation of a product of two Gegenbauer functions of the first kind in the trigonometric context.
\begin{cor}
\label{cor18}
Let $x_1,x_2\in\CC\setminus((-\infty,-1]\cup[1,\infty))$, $\lambda,\alpha\in\CCast$, such that 
$\Re(\lambda+2\alpha)>0$, $\Re\lambda<4$. Then
\begin{eqnarray}
&&\hspace{0.0cm}{\sf C}_\lambda^\alpha(x_1){\sf C}_\lambda^\alpha(x_2)=\frac{-2^\lambda\sqrt{\pi}\,\Gamma(\frac12\lambda+\alpha+\frac12)}{\Gamma(\lambda)\Gamma(\alpha)\Gamma(1-\frac{\lambda}{2})(1-x_1^2)^{\frac12\alpha-\frac14}(1-x_2^2)^{\frac12\alpha-\frac14}}
\int_{\frac{x_1+x_2}{1+x_1x_2}}^1\dd x_3\,{\sf C}_\lambda^\alpha(x_3)(1-x_3^2)^{\frac12\alpha-\frac54}P_{\lambda+\alpha-\frac12}(B),\nonumber\\
\end{eqnarray}
where $B$ is defined in \eqref{Bdef}, namely,
\begin{equation}
B=\frac{2-x_1^2-x_2^2-x_3^2+x_1^2x_2^2x_3^3}{2\sqrt{(1-x_1^2)(1-x_2^2)(1-x_3^2)}}.
\end{equation}
\end{cor}
\begin{proof}
Start with \eqref{prodPPcut} and set $\beta=\frac12$ and substitute \eqref{rJhGeg}. In order to obtain the restriction on the parameters, perform Taylor expansion of the integrand about the singularity at $x_3\sim 1$. This completes the proof.
\end{proof}

\noindent Note the interesting consequential formula if $\alpha=\frac12$ using \cite[\href{http://dlmf.nist.gov/18.7.E9}{(18.7.9)}]{NIST:DLMF}, $C_\nu^\frac12(x)={\sf P}_\nu(x)$,  namely
\begin{eqnarray}
&&\hspace{-0.3cm}{\sf P}_\lambda(x_1){\sf P}_\lambda(x_2)=\frac{-2^\lambda\Gamma(\frac12\lambda+1)}{\Gamma(\lambda)\Gamma(1-\frac{\lambda}{2})}
\int_{\frac{x_1+x_2}{1+x_1x_2}}^1\dd x_3\,{\sf P}_\lambda(x_3)(1-x_3^2)^{-1}P_{\lambda}\left(\frac{2-x_1^2-x_2^2-x_3^2+x_1^2x_2^2x_3^3}{2\sqrt{(1-x_1^2)(1-x_2^2)(1-x_3^2)}}\right),
\end{eqnarray}
where ${\sf P}_\lambda$ is the Ferrers function of the first kind. This formula is a special case $\mu=0$ of the following integral representation of a product of two Ferrers functions of the first kind.
\begin{cor}Let $x_1,x_2\in\CC\setminus((-\infty,-1]\cup[1,\infty))$, $\nu,\mu\in\CCast$, such that 
$\Re(\mu+\nu)>-1$, $\Re(\mu-\nu)>-2$. Then
\begin{eqnarray}
&&\hspace{-1.5cm}{\sf P}_\nu^{-\mu}(x_1)
{\sf P}_\nu^{-\mu}(x_2)=\frac{2^{1-\mu}\sqrt{\pi}}{\Gamma(\frac12(\mu-\nu))\Gamma(\frac12(\nu+\mu+1))}\nonumber\\
&&\hspace{1.5cm}\times\int_{\frac{x_1+x_2}{1+x_1x_2}}^1\dd x_3\,{\sf P}_\nu^{-\mu}(x_3)(1-x_3^2)^{-1}P_{\nu}\left(\frac{2-x_1^2-x_2^2-x_3^2+x_1^2x_2^2x_3^3}{2\sqrt{(1-x_1^2)(1-x_2^2)(1-x_3^2)}}\right).
\end{eqnarray}
\end{cor}
\begin{proof}
Start with Corollary \ref{cor18} and replace the Gegenbauer functions of the first kind in the trigonometric context  with Ferrers functions of the first kind using \eqref{ClmFerP}. By replacing parameters accordingly the result follows.
In order to obtain the restriction on the parameters, perform Taylor expansion of the integrand about the singularity at $X_3\sim 1$. This completes the proof.
\end{proof}

\begin{thm}
\label{thm20}
Let $z_1,z_2\in\CC\setminus(-\infty,1]$, $\lambda,\alpha\in\CCast$, such that $\Re(\lambda+2\alpha)>0$. Then
\begin{eqnarray}
&&\hspace{-2.5cm}C_\lambda^\alpha(z_1)C_\lambda^\alpha(z_2)=\frac{\cos(\pi(\lambda+\alpha))\Gamma(\lambda+\alpha+\frac12)}{2^{\alpha-\frac12}\sqrt{\pi}\,\Gamma(\lambda+1)\Gamma(\alpha)(z_1z_2)^{\lambda+\alpha+\frac12}(z_1^2-1)^{\frac12\alpha-\frac14}(z_2^2-1)^{\frac12\alpha-\frac14}}\nonumber\\
&&\hspace{0.5cm}\times\int_{1}^{\frac{z_1z_2+1}{z_1+z_2}}\dd z_3\,C_\lambda^\alpha(z_3)\frac{(z_3^2-1)^{\frac12\alpha-\frac54}}{z_3^{\lambda+\alpha+\frac12}}(C-1)^{-\lambda-\alpha-\frac12},
\end{eqnarray}
where $C=C(z_1,z_2,z_3)$ is given by \eqref{Cdef}.
\end{thm}
\begin{proof}
Start with Theorem \ref{thm14}, set $\beta=-\frac12$ and utilize \eqref{P2g}.  This converts the product of two Jacobi functions of the first kind on the left-hand side to products of Gegenbauer functions of the first kind and the Jacobi function of the first kind in the integrand to a Gegenbauer function of the first kind using \eqref{rJmhGeg}. Now set $\{\gamma,\alpha\}\mapsto\{\frac12\lambda,\alpha-\frac12\}$. The Gegenbauer function of the first kind with argument $C$ can be simplified using Lemma \ref{lem1b}.
In order to obtain the restriction on the parameters, perform an asymptotic analysis of the integrand about the singularity at $z_3\sim 1$. 
Simplifying completes the proof.
\end{proof}

\noindent By starting with Theorem \ref{thm20} and replacing the Gegenbauer functions of the first kind with associated Legendre functions of the first kind using \eqref{ClmLegP}, we obtain the following result.
\begin{cor}
\label{cor30}
Let $z_1,z_2\in\CC\setminus(-\infty,1]$, $\nu,\mu\in\CCast$, such that $\Re(\nu+\mu)>-3$. Then
\begin{eqnarray}
&&\hspace{-3.0cm}P_\nu^{-\mu}(z_1)P_\nu^{-\mu}(z_2)=\frac{\sin(\pi\nu)\Gamma(\nu+1)}{\pi\Gamma(\nu+\mu+1)(z_1z_2)^{\nu+1}}\int_{1}^{\frac{1+z_1z_2}{z_1+z_2}}\dd z_3\,P_\nu^{-\mu}(z_3)\frac{(C-1)^{-\nu-1}}{z_3^{\nu+1}(1-z_3^2)},
\end{eqnarray}
where $C$ is given by \eqref{Cdef}.
\end{cor}
\begin{proof}
First start with Theorem \ref{thm20} and then replace the Gegenbauer functions of the first kind on both sides of the equation with associated Legendre functions of the first kind using \eqref{ClmLegP}. In order to obtain the restriction on the parameters, perform an asymptotic analysis of the integrand about the singularity at $z_3\sim 1$. 
This completes the proof. 
\end{proof}

\section{Single function integral representations and Bateman-type expansions
\label{sec:single_func_ints}}

One can derive a Laplace-type integral representation in the flavor of \cite[(4.3)]{FlenstedJensenKoornwinder73} for the Jacobi function of the second kind by taking the asymptotic limit of both sides of \eqref{DurandQQe} as $y\to\infty$. 
\begin{cor}
\label{LapQthm}
Let $\lambda,\mu,\gamma,\alpha,\beta\in\CCast$, $z\in\CC\setminus(-\infty,1]$. Then one has the following Laplace-type integral representation of the Jacobi function of the second kind,
\begin{eqnarray}
&&\hspace{0.0cm}(z-1)^{\lambda+\frac12\mu}(z+1)^{\frac12\mu}Q_{\gamma-\lambda-\mu}^{(\alpha+2\lambda+\mu,\beta+\mu)}(z)=\frac{2^{2\beta+\lambda+\mu}\Gamma(\beta)\Gamma(\lambda+1)\Gamma(\mu+1)\Gamma(\alpha\!+\!\gamma\!+\!1)\Gamma(\alpha\!+\!\beta\!+\!\gamma\!+\!1)}{\Gamma(\alpha\!-\!\beta\!+\!\lambda)\Gamma(2\beta\!+\!\mu)\Gamma(\gamma\!-\!\lambda\!-\!\mu\!+\!1)\Gamma(\alpha\!+\!\beta\!+\!\gamma\!+\!\lambda\!+\!\mu\!+\!1)}\nonumber\\
&&\hspace{1.0cm}\times
\int_1^\infty
\int_1^\infty 
P_\lambda^{(\alpha-\beta-1,\beta+\mu)}(2w^2-1)C_\mu^{\beta}(x)
\frac{w^{2\beta+\mu+1}(w^2-1)^{\alpha-\beta-1}(x^2-1)^{\beta-\frac12}}{(z+\sqrt{z^2-1}xw+\frac12(z-1)(w^2-1))^{\alpha+\beta+\gamma+1}}
\dd w \,\dd x.
\label{LaplaceQQ}
\end{eqnarray}
\end{cor}
\begin{proof}
Starting with \eqref{DurandQQe}, taking the asymptotic limit as $y\to\infty$ of both sides using \eqref{asympQinf} and simplifying completes the proof. 
\end{proof}

\noindent The simpler case of the Laplace-type integral with $\lambda=\mu=0$ was give in \cite[(2.15)]{Durand78}. 
\vspace*{1ex}

\noindent 
Given this Laplace-type integral representation, we are able to obtain a Bateman-type sum for a product of two Jacobi functions of the second kind. 
\begin{thm}
\label{thm01s}
Let $x,y\in\CCast\setminus(-\infty,1]$, $\lambda,\mu,\alpha,\beta\in\CC$. Then
\begin{eqnarray}
&&\hspace{-0.0cm}Q_{\gamma-\lambda-\mu}^{(\alpha+2\lambda+\mu,\beta+\mu)}(x)
Q_{\gamma-\lambda-\mu}^{(\alpha+2\lambda+\mu,\beta+\mu)}(y)
=\frac{\Gamma(\alpha+\gamma+\lambda+1)\Gamma(\beta+\gamma-\lambda+1)}{2\,\Gamma(\alpha+\beta+2\gamma+2)}
\left(\frac{2}{x+y}\right)^{\alpha+\beta+\gamma+\lambda+\mu+1}\nonumber\\
&&\hspace{1.0cm}\times\sum_{k=0}^\infty
\frac{(\gamma-\lambda-\mu+1)_k(\alpha+\beta+\gamma+\lambda+\mu+1)_k}{(\alpha+\beta+2\gamma+2)_kk!}
\left(\frac{2}{x+y}\right)^k
Q_{\gamma-\lambda-\mu+k}^{(\alpha+2\lambda+\mu,\beta+\mu)}\left(\frac{1+xy}{x+y}\right).
\label{prodQQsum}
\end{eqnarray}
\end{thm}
\begin{proof}
First start with Durand's integral representation for a product of two Jacobi functions of the second kind \eqref{DurandQQe}
and then expanding the Jacobi function of the second kind as an infinite series over $k\in\N_0$ using its hypergeometric representation \eqref{dJsk2}.
One then is able to evaluate the double integral for each term of the infinite series using the
Laplace-type integral representation for the Jacobi function of the second kind \eqref{LaplaceQQ} which completes the proof. 
\end{proof}

\noindent 
By starting with \eqref{prodQQsum} we are able to derive a Bateman-type sum for a product of two Gegenbauer functions of the second kind.
\begin{cor}Let $\gamma,\alpha,\mu\in\CC$, $x,y\in\CC\setminus(-\infty,1]$. Then one has the following Bateman-type sum for Gegenbauer functions of the second kind, namely 
\begin{eqnarray}
&&\hspace{-1cm}D_{\gamma-\mu}^{\alpha+\mu}(x)D_{\gamma-\mu}^{\alpha+\mu}(y)=\frac{\expe^{i\pi(\alpha+\mu)}\Gamma(2\alpha+\gamma+\mu)}{2^{4\alpha+2\gamma+2\mu}\Gamma(\alpha+\mu)\Gamma(\alpha+\gamma+1)}\nonumber\\
&&\hspace{2cm}\times\sum_{k=0}^\infty
\frac{(\gamma-\mu+1)_k(\alpha+\gamma+\frac12)_k}{(2\alpha+2\gamma+1)_k\,k!}\left(\frac{2}{x+y}\right)^{2\alpha+\gamma+\mu+k}
D_{\gamma-\mu+k}^{\alpha+\mu}\left(\frac{1+xy}{x+y}\right).
\end{eqnarray}
\end{cor}
\begin{proof}
Start with \eqref{prodQQsum} and set $\lambda=0$ followed by $\beta,\alpha\mapsto \alpha-\frac12$. Then utilizing \eqref{GegDsym} three times completes the proof.
\end{proof}
\noindent 
Therefore, one also has a Bateman-type sum for associated Legendre functions of the second kind using \eqref{DlmLegQ},
\begin{eqnarray}
&&\hspace{-1.1cm}Q_{\nu+\alpha}^{\mu+\alpha}(x)Q_{\nu+\alpha}^{\mu+\alpha}(y)=\frac{\expe^{i\pi(\mu+\alpha)}\sqrt{\pi}\,\Gamma(\nu+\mu+2\alpha+1)}{2^{2\nu+2\alpha+2}\Gamma(\nu+\alpha+\frac32)}\nonumber\\
&&\hspace{2cm}\times\sum_{k=0}^\infty
\frac{(\nu-\mu+1)_k(\nu+\alpha+1)_k}{(2\nu+2\alpha+2)_k\,k!}\left(\frac{2}{x+y}\right)^{\nu+\alpha+k+1}
Q_{\nu+\alpha+k}^{\mu+\alpha}\left(\frac{1+xy}{x+y}\right).
\end{eqnarray}

\noindent By examining the behavior of Theorem \ref{bigQthm} near the singularity at $z_2\sim 1$, we can obtain the following integral representation of a Jacobi function of the second kind.
\begin{thm}
\label{thm33}
Let $z\in\CC\setminus(-\infty,1]$, $\alpha,\beta,\gamma\in\CCast$, such that $\Re\alpha>0$, $\Re\gamma,\Re(\gamma+\beta)>-1$,  $\beta\ne-\frac12$, $\alpha+\gamma,\alpha+\beta+\gamma\not\in-\N$. Then
\begin{eqnarray}
&&\hspace{-1.5cm}Q_\gamma^{(\alpha,\beta)}(2z^2-1)=\frac{\expe^{-i\pi(\beta+\frac12)}2^{\alpha+\beta+1}\sqrt{\pi}\,\Gamma(\beta+\frac12)\Gamma(\alpha+\gamma+1)\Gamma(\alpha+\beta+\gamma+1)z^{\alpha-\beta-1}}{\Gamma(\alpha)\Gamma(\gamma+1)\Gamma(\alpha+\beta)\Gamma(\beta+\gamma+1)(z^2-1)^\alpha}\nonumber\\
&&\hspace{1.8cm}\times\int_{z}^\infty\dd w\, Q_\gamma^{(\alpha,\beta)}(2w^2-1)w^{\alpha+\beta}\left(E^2-1\right)^{\frac12(\alpha-\beta-1)}D_{\alpha-\beta-1}^{\beta+\frac12}\left(\frac{E}{\sqrt{E^2-1}}\right),\label{QsingD}\\
&&\hspace{1.1cm}=\frac{\expe^{i\pi\beta}2^{\alpha+1}\Gamma(\alpha+\gamma+1)\Gamma(\alpha+\beta+\gamma+1)z^{\alpha-\beta-1}}{\Gamma(\alpha)\Gamma(\gamma+1)\Gamma(\alpha+\beta)\Gamma(\beta+\gamma+1)(z^2-1)^{\alpha}}\nonumber\\
&&\hspace{1.8cm}\times\int_{z}^\infty\dd w\, Q_\gamma^{(\alpha,\beta)}(2w^2-1)w^{\alpha+\beta}\left(E^2-1\right)^{\frac12(\alpha-1)}Q_{\alpha-1}^{\beta}\left(\frac{E}{\sqrt{E^2-1}}\right),\label{QsingQ}
\end{eqnarray}
where
\begin{equation}
E:=E(z,w):=\frac{z^2+w^2}{2wz}.
\end{equation}
\end{thm}
\begin{proof}
Start with Theorem \ref{bigQthm} and examine the behavior of both sides if $z_2$ is near the singularity at unity, namely let $z_2=1+\epsilon$ with $\epsilon\ll 1$ and utilize \eqref{JacQsing1}, noting that then $2z_2^2-1\sim 1+4\epsilon$. This asymptotic analysis requires $\Re\alpha>0$. Let $y_3=w$. As $\epsilon\to 0^{+}$ then 
$A(y_1,y_2,w)\sim E(z,w)$ and the lower bound of integration becomes $z_1:=z$. Then from \eqref{JacobiQQD} we obtain \eqref{QsingD} and from \eqref{JacQQQ} we obtain \eqref{QsingQ}, where we have used \eqref{ptyQmu}. Performing an asymptotic analysis as $w\to\infty$ of the integrand and removing the singularities of the gamma functions in the pre-factor completes the proof. 
\end{proof}

\begin{rem}
The integral representation for the Jacobi function of the second kind given in Theorem~\ref{thm33} (taking into account the restriction on parameters), exists for $\gamma\in\N_0$ as long as 
$\Re\alpha>0$ and $\Re\beta<-1$.
\end{rem}

\noindent By examining the behavior of the product of two Jacobi functions of the first kind as $z_2\to\infty$ we can obtain an integral representation of a Jacobi function of the first kind.
\begin{thm}
\label{thm32}
Let $z\in\CC\setminus(-\infty,1]$, $\alpha,\beta,\gamma\in\CCast$, such that   $\Re(\alpha+\gamma)>0$, $\Re(\alpha+\beta+\gamma)>0$, $\Re(-\alpha-\beta-2\gamma-1)>0$, $\alpha+2\gamma+2,\alpha+2\beta+2\gamma+2\not\in-\N_0$. Then
\begin{eqnarray}
&&\hspace{-1.0cm}P_\gamma^{(\alpha,\beta)}(2z^2-1)=\frac{\pi\sin(\pi(\alpha+2\beta+2\gamma))\Gamma(\alpha+2\gamma+2)\Gamma(\alpha+2\beta+2\gamma+2)}{2^{2\alpha+2\beta+4\gamma+2}\sin(\pi\gamma)\sin(\pi(\beta+\gamma))\Gamma(\gamma+1)\Gamma(\alpha+\gamma+1)\Gamma(\beta+\gamma+1)\Gamma(\alpha+\beta+\gamma+1)}\nonumber\\
&&\hspace{-0.5cm}\times\frac{1}{z^{\alpha+2\beta+2\gamma+2}(z^2-1)^{\frac12\alpha}}\int_1^{z}\dd z\,P_\gamma^{(\alpha,\beta)}(2w^2-1)\frac{(w^2-1)^{\frac12(\alpha-2)}}{w^{\alpha+2\gamma+1}}(u^2-1)^{-\alpha-\beta-2\gamma-\frac32}C_{\alpha+2\gamma+1}^{-(\alpha+\beta+2\gamma+1)}(u),
\label{JacPsing}
\end{eqnarray}
where
\begin{equation}
u:=u(z,w):=\frac{2z^2w^2-z^2-w^2}{2zw\sqrt{(z^2-1)(w^2-1)}}.
\end{equation}
\end{thm}
\begin{proof}
Start with Theorem \ref{thm14} and set $z=z_1$. Then examine the asymptotic behavior as $|z_2|\to\infty$ on both sides of the equation. Using \eqref{asympPinf2} on the left-hand side requires that $\Re(-\alpha-\beta-2\gamma-1)>0$. On the right-hand side as $z\to\infty$, $C\mapsto u$ and all terms proportional to $z_2$ cancel. The remaining restrictions on parameters come from performing an asymptotic analysis of the integrand near the singularity at $w=1$ and removing the singularities of the numerator gamma functions in the prefactor. This completes the proof. 
\end{proof}

\begin{rem}
Due to the restriction on parameters, note that the integral representation of the Jacobi function of the first kind given in Theorem \ref{thm32}, does not exist for $\gamma=n\in\N_0$. So this integral representation does not correspond to an integral representation for Jacobi polynomials.
\end{rem}

\noindent We can also obtain an integral representation of a Jacobi function of the first kind in the trigonometric context by examining the asymptotic behavior of Theorem 
\ref{thm23} as $x_2\sim -1$.
\begin{thm}
\label{thm33b}
Let $\alpha,\beta,\gamma\in\CCast$, $x\in\CC\setminus((-\infty,-1]\cup(1,\infty))$, $\Re\beta>0$, $\Re(\beta+\gamma)<0$, $\Re(\alpha+\gamma)>-1$. Then
\begin{eqnarray}
&&\hspace{-3.0cm}{\sf P}_\gamma^{(\alpha,\beta)}(2x^2-1)=
\frac{2^{2\beta}\sin(\pi(\beta+\gamma))\Gamma(\alpha+2\gamma+2)\Gamma(\beta+\gamma+1)\Gamma(\alpha+\beta+\gamma+1)}{\sin(\pi\gamma)\Gamma(\gamma+1)\Gamma(\alpha+\gamma+1)\Gamma(\alpha+2\beta+2\gamma+1)x^{2\beta}(1-x^2)^{\frac12\alpha}}\nonumber\\
&&\hspace{0.0cm}\times\int_{x}^{1}\dd t\,{\sf P}_\gamma^{(\alpha,\beta)}(2t^2-1)t(1-t^2)^{\frac12(\alpha-2)}
\left(G^2-1\right)^{\beta-\frac12}C_{\alpha+2\gamma+1}^\beta(G),\end{eqnarray}
where 
\begin{equation}
G:=G(x,t):=\frac{2-x^2-t^2}{2\sqrt{(1-x^2)(1-t^2)}}.
\end{equation}
\end{thm}
\begin{proof}
Start with Theorem \ref{thm23} and examine the asymptotic behavior near the singularity at $x_2\sim-1$. Take $x_2\sim -1+\sqrt{\epsilon}$, and use \eqref{asympPcutn1}. On the right-hand side $B\mapsto G$ and all the factors of $\epsilon$ cancel. The restrictions on the parameters comes from performing an asymptotic analysis at the endpoints of the integral and from removing the singularities of the gamma function. Simplification completes the proof.
\end{proof}

\begin{rem}
Due to the restriction on parameters, note that the integral representation of the Jacobi function of the first kind given in Theorem \ref{thm33b}, does not exist for $\gamma=n\in\N_0$. So this integral representation cannot correspond to an integral representation for Jacobi polynomials.
\end{rem}

\noindent By analyzing the asymptotic behavior as $z_2\sim 1$ in \eqref{LegQQint}, we can obtain an integral representation of an associated Legendre function of the second kind.
\begin{thm}
\label{thm39}
Let $z\in\CC\setminus(-\infty,1]$, $\nu,\mu\in\CCast$, such that $\Re\mu>0$, $\Re(\nu-\mu)>-1$. Then
\begin{eqnarray}
&&\hspace{-5cm}Q_\nu^\mu(z)=\frac{\Gamma(\nu+\mu+1)}{\Gamma(2\mu)\Gamma(\nu-\mu+1)(z^2-1)^{\frac12\mu}}\int_z^\infty
\dd w\,Q_\nu^{\mu}(w)
\frac{(w-z)^{2\mu-1}}{(w^2-1)^{\frac12\mu}}.
\end{eqnarray}
\end{thm}
\begin{proof}
Start with \eqref{LegQQint} and perform an asymptotic analysis of $Q_\nu^\mu(z_2)$ as $z_2\sim 1+\epsilon$ using \eqref{asympQ1mp}.
Simplification completes the proof.
\end{proof}

\begin{thm}
\label{thm35}
Let $z\in\CC\setminus(-\infty,1]$, $\lambda,\alpha\in\CCast$, such that $\Re(\lambda+2\alpha)>0$, $\Re(2\lambda+2\alpha)<0$. Then
\begin{eqnarray}
&&\hspace{-2.8cm}C_\lambda^\alpha(z)=\frac{\sin(2\pi(\lambda+\alpha))\Gamma(2\lambda+2\alpha+1)(z^2-1)^{\frac12(\lambda+1)}}{\sin(\pi\lambda)\Gamma(\lambda+1)\Gamma(\lambda+2\alpha)}\nonumber\\
&&\hspace{-1.0cm}\times\int_1^z\dd w\,C_\lambda^\alpha(w)\frac{(w^2-1)^{\frac12(\lambda+2\alpha-2)}}{\left(2z^2w^2-z^2-w^2-2zw(z^2-1)^\frac12(w^2-1)^\frac12\right)^{\frac12(2\lambda+2\alpha+1)}}.
\end{eqnarray}
\end{thm}
\begin{proof}
Start with Theorem \ref{thm32} and set $\beta=-\frac12$.  Then use \eqref{rJmhGeg} to write the Jacobi functions of the first kind in terms of Gegenbauer functions of the first kind. The constraints on the parameters are inherited from Theorem \ref{thm20}. Simplification completes the proof. 
\end{proof}
\begin{thm}
Let $z\in\CC\setminus(-\infty,1]$, $\nu,\mu\in\CCast$, such that $\Re\nu<-\frac12$, $\Re(\nu+\mu)>-2$. Then
\begin{eqnarray}
&&\hspace{-0.5cm}P_\nu^{-\mu}(z)=\frac{\sin(2\pi\nu)\Gamma(2\nu+2)(z^2-1)^{\frac12(\nu+1)}}{\sin(\pi(\mu-\nu))\Gamma(\nu+\mu+1)\Gamma(\nu-\mu+1)}\int_1^z\dd w\,\frac{P_\nu^{-\mu}(w)(w^2-1)^{\frac12(\nu+1)}}{\left(\left(zw-(z^2-1)^\frac12(w^2-1)^\frac12\right)^2-1\right)^{\nu+1}}
.
\end{eqnarray}
\end{thm}
\begin{proof}
Start with Theorem \ref{thm35} and replace the Gegenbauer functions of the first kind with associated Legendre functions of the first kind. Rearranging parameters and inheriting parameter constraints from Corollary \ref{thm32} completes the proof.
\end{proof}

\begin{thm}\label{thm37}
Let $\lambda,\alpha\in\CCast$, $x\in\CC\setminus((-\infty,-1]\cup(1,\infty))$, 
$\Re \lambda<0$,
$\Re(\lambda+2\alpha)>0$. Then
\begin{eqnarray}
&&\hspace{-0.7cm}{\sf C}_\lambda^{\alpha}(x)=\frac{-2\tan(\frac12\pi\lambda)\Gamma(\frac12\lambda+\alpha+\frac12)\Gamma(\frac12\lambda+1)}{\Gamma(\frac12\lambda+\alpha)\Gamma(\frac12(\lambda+1))}\int_x^1\dd t\,{\sf C}_\lambda^\alpha(t)(1-t^2)^{\frac12\alpha-\frac54}P_{\lambda+\alpha-\frac12}\left(\frac{2-x^2-t^2}{2\sqrt{(1-x^2)(1-t^2)}}\right).
\end{eqnarray}
\end{thm}
\begin{proof}
Start with Theorem \ref{thm33b} and set $\beta=\frac12$. Then use \eqref{rJhGeg} to write the Jacobi functions of the first kind in the trigonometric context on both sides of the equation in terms of  Gegenbauer functions of the first kind in the trigonometric context. Then make the replacements $(\gamma,\alpha)\mapsto(\frac12\lambda-\frac12,\alpha-\frac12)$. To determine the restrictions on the parameters perform an asymptotic analysis as the integration variable approaches the endpoints of integration and removing the singularities of the numerator gamma functions. 
\end{proof}

\begin{thm}Let $\nu,\mu\in\CCast$, $x\in\CC\setminus((-\infty,-1]\cup(1,\infty))$, $\Re(\mu+\nu)>-3$, $\Re(\mu-\nu)>0$. Then
\begin{eqnarray}
&&\hspace{-2cm}{\sf P}_\nu^{-\mu}(x)=\frac{2(\nu-\mu)\Gamma(\frac{\mu-\nu+3}{2})\Gamma(\frac{\nu+\mu+2}{2})}{(\nu-\mu-1)\Gamma(\frac{\mu-\nu+2}{2})\Gamma(\frac{\nu+\mu+1}{2})}\int_x^1\frac{\dd t}{1-t^2}\,{\sf P}_\nu^{-\mu}(t)P_\nu\left(\frac{2-x^2-t^2}{2\sqrt{(1-x^2)(1-t^2)}}\right).
\end{eqnarray}
\end{thm}
\begin{proof}
Start with Theorem \ref{thm37} and use the identity \eqref{ClmFerP} to convert the Gegenbauer functions of the first kind on both sides of the equation in terms of Ferrers functions of the first kind. The restrictions on the parameter values comes from performing an asymptotic analysis of the integrand as the the integration variable approaches the endpoints. Simplification completes the proof.
\end{proof}

\section{Nicholson-type formulas\label{sec:math_appl}}

Nicholson-type formulas generalize the relation $\expe^{ix}\expe^{-ix}=\sin{}^2x+\cos{}^2x$ for trigonometric functions to other special functions.  The original such formula, derived by J.~W.~Nicholson in 1910 for the case of Bessel functions \cite{Nicholson1910,Nicholson1911}, has been used to investigate a number of properties of those functions \cite[\S\S 13.73-13.75]{Watson}. In 1971 Durand \cite{Durand75} obtained Nicholson-type expressions for Gegenbauer and Legendre function which led through a confluent limit to a generalization of Nicholson's formula for Bessel functions, and in a different limit, to corresponding results for Hermite functions. The results on Legendre functions were used in \cite{Durand75} to derive various monotonicity properties of, and bounds on, those functions as well as some asymptotic results. However, the analysis is clearly not complete compared to that of the Bessel functions, and the Gegenbauer and Hermite functions were not considered in detail. Durand extended his results in 1978 to the case of Jacobi functions in a double-integral form \cite[(2.13), (2.14)]{Durand78}, and to the case of Laguerre functions using a confluent limit of his kernel expression for the product of two Jacobi functions of the second kind \cite[(4.12)]{Durand78}.

In this section we will derive a single-integral Nicholson-type formula for Jacobi functions from the corrected kernel form of the $QQ$ product formula for Jacobi functions of the second kind in Theorem \ref{bigQthm}. This leads directly to similar Nicholson-type formulas for Legendre and Gegenbauer functions which were derived originally from the $DD$ product formula for two Gegenbauer functions of the second kind, \cite[(13)]{Durand75}.  We then use an appropriate confluent limit on the $QQ$ formula to obtain the Nicholson-type expression for Laguerre functions. Since the error in Durand's kernel formula was in the  lower limit of integration in the $z_1,\,z_2$ version of  Theorem \ref{bigQthm} which he gave as 1 instead of $z_1z_2+\sqrt{z_1^2-1}\sqrt{z_2^2-1}$ (0 instead of $t_1+t_2$ in the hyperbolic variables used in \cite{Durand78}), and the confluent limit drives $z_1\rightarrow1$, $z_z\rightarrow 1$, the results for the Laguerre functions given in \cite[(4.10), (4.12)]{Durand78} are unchanged.

\subsection{Nicholson-type results for Jacobi functions}

Our preceding analysis leads to the following Nicholson-type  formula for Jacobi functions in the trigonometric context which is in kernel form and involves an integral over a Jacobi function of the second kind multiplied by  either a Gegenbauer or associated Legendre function of the second kind.

\begin{thm}Let $\gamma,\alpha,\beta\in\CC$, $x\in\CC\setminus((-\infty,1]\cup[1,\infty)])$. Then
\begin{eqnarray}
&&\hspace{-0.6cm}[{\sf Q}_\gamma^{(\alpha,\beta)}(x)]^2
+\frac{\pi^2}{4}[{\sf P}_\gamma^{(\alpha,\beta)}(x)]^2=\frac{\expe^{-i\pi (\beta+\frac12)}\sqrt{\pi}\,2^{\alpha+2\beta}\Gamma(\beta+\frac12)\Gamma(\alpha+\gamma+1)}{\Gamma(\gamma+1)\Gamma(\alpha+\beta)(1-x)^{2\alpha}}\nonumber\\
&&\hspace{-0.1cm}\times\int_1^\infty Q_\gamma^{(\alpha,\beta)}(w)(w-1)^{\frac{\alpha-\beta-1}{2}}(w+1)^{\beta}(1-2x^2+w)^{\frac{\alpha-\beta-1}{2}}D_{\alpha-\beta-1}^{\beta+\frac12}\left(\frac{1+2x+w}{\sqrt{(w-1)(1-2x^2+w)}}\right)\dd w\\
&&\hspace{-0.3cm}=\frac{\expe^{-i\pi\beta}2^{\alpha+\frac{\beta}{2}}\Gamma(\alpha+\gamma+1)}{\Gamma(\gamma+1)\Gamma(\alpha+\beta)(1-x)^{2\alpha}(1+x)^\beta}\nonumber\\
&&\hspace{0.4cm}\times\int_1^\infty Q_\gamma^{(\alpha,\beta)}(w)(w-1)^{\frac{\alpha-1}{2}}(w+1)^{\frac{\beta}{2}}(1-2x^2+w)^{\frac{\alpha-1}{2}}Q_{\alpha-1}^{\beta}\left(\frac{1+2x+w}{\sqrt{(w-1)(1-2x^2+w)}}\right)\dd w.
\label{NichJacQ}
\end{eqnarray}
\end{thm}
\begin{proof}
The key result which provides Nicholson-type relations for Jacobi functions in the trigonometric context is \cite[(2.12)]{Durand78}, namely 
\begin{equation}
Q_\gamma^{(\alpha,\beta)}(x+i0)
Q_\gamma^{(\alpha,\beta)}(x-i0)=
[{\sf Q}_\gamma^{(\alpha,\beta)}(x)]^2+\frac{\pi^2}{4}[{\sf P}_\gamma^{(\alpha,\beta)}(x)]^2.
\end{equation}
We can directly apply this to Theorem \ref{bigQthm}. First make the replacement $2y_1^2-1,\,2y_2^2-1\mapsto z_1,z_2$, then make the change of integration $2y_3^2-1\mapsto w$. Using these replacements, the lower bound of integration becomes unity and the upper bound of integration remains the same. Then making the replacements $z_1\mapsto x+i0$, $z_2\mapsto x-i0$ and using the rules \eqref{trighyprules} completes the proof.
\end{proof}

By starting with \eqref{NichJacQ}, and  setting $\beta=\alpha$ in the above theorem and taking advantage of Lemma \ref{lem3}, we can easily obtain a Nicholson-type formula for Ferrers functions. This leads to the following result.

\begin{thm}
\label{NichFer}Let $\nu,\mu\in\CC$, $x\in\CC\setminus((-\infty,1]\cup[1,\infty))$. Then
\begin{eqnarray}
&&[{\sf Q}_\nu^{-\mu}(x)]^2+\frac{\pi^2}{4}[{\sf P}_\nu^{-\mu}(x)]^2=\frac{\sqrt{\pi}\expe^{i\pi\mu}}{2^\mu\Gamma(\mu+\frac12)(1-x^2)^{\mu}}
\int_1^\infty Q_\nu^{-\mu}(w)\frac{(w-1)^{\frac{\mu-1}{2}}}{(w+1)^{\frac{\mu}{2}}}(1-2x^2+w)^{\mu-\frac12}\,\dd w.
\end{eqnarray}
\end{thm}
\begin{proof}
Start with \eqref{NichJacQ} and set $\beta=\alpha$. Then the associated Legendre function of the second kind in the integrand becomes an elementary function using Lemma \ref{lem3}. The Jacobi functions of the first and second kind in the trigonometric context on the left-hand side become Ferrers functions of the first and second kind using \eqref{relJacPFerP}, \eqref{FerQsym}. Combining common terms on the left-hand side and simplifying completes the proof.
\end{proof}

\noindent An obvious consequence of the Theorem \ref{NichFer} by setting $\mu\to0$ is the following Nicholson-type formula for Legendre functions in the trigonometric context :
\begin{equation}
[{\sf Q}_\nu(x)]^2+\frac{\pi^2}{4}[{\sf P}_\nu(x)]^2=
\int_1^\infty Q_\nu(w)\frac{(1-2x^2+w)^{-\frac12}}{(w-1)^{\frac{1}{2}}}\,\dd w
\end{equation}
where the integral is over the Legendre function of the second kind in the hyperbolic context. This Nicholson-type formula was  derived in \cite[p.153,\,(18a)]{Durand75} and \cite[3.12:)]{Durand78}, and is  included  in \cite[\href{http://dlmf.nist.gov/18.17.E7}{(18.17.7)}]{NIST:DLMF} along with the corresponding result for Hermite functions from \cite[(52)]{Durand75} in \cite[\href{http://dlmf.nist.gov/18.17.E7}{(18.17.8)}]{NIST:DLMF}. The degrees in those formulas are not required to be integer.  By using the relations between Ferrers functions and Gegenbauer functions in the trigonometric context \eqref{ClmFerP}, \eqref{DlmFerQ}, we can obtain from Theorem \ref{NichFer} a Nicholson-type integral formula for Gegenbauer functions in the trigonometric context, the form in which the results were originally derived in \cite{Durand75}.
\begin{thm}Let $\lambda,\alpha\in\CC$, $x\in\CC\setminus((-\infty,1]\cup[1,\infty))$. Then
\begin{eqnarray}
&&[{\sf C}_\lambda^{\alpha}(x)]^2+[{\sf D}_\lambda^{\alpha}(x)]^2=\frac{\expe^{-i\pi\alpha}\Gamma(\lambda+2\alpha)}{2^{2\alpha-3}\Gamma(\alpha)^2\Gamma(\lambda+1)(1-x^2)^{2\alpha-1}}
\int_1^\infty D_\lambda^{\alpha}(w)(w-1)^{\alpha-1}(1-2x^2+w)^{\alpha-1}\,\dd w.
\end{eqnarray}
\end{thm}
\begin{proof}
Start with Theorem \ref{NichFer} and convert the Ferrers functions on the left-hand side to Gegenbauer functions in the trigonometric context using Theorem \ref{GegFer}. Shifting parameters and simplifying completes the proof.
\end{proof}

\noindent This result is a form of the Nicholson-type formula for Gegenbauer functions derived directly from the $DD$ product formula for the Gegenbauer function of the second kind \cite[(18a)]{Durand75}.
\medskip


\subsection{Nicholson-type results for Laguerre functions}


\noindent As described in \cite[\S4]{Durand78}, one can obtain Laguerre functions as a confluent limit of Jacobi functions. As with Jacobi functions, there are  Laguerre functions in the hyperbolic context  of the first and second kind, $L_\gamma^\alpha$, $N_\gamma^\alpha$, which are well-defined for $z\in\CC\setminus[0,\infty)$ with $L_\gamma^\alpha(z)$ entire and $N_\gamma^\alpha(z)$ cut along the positive $z$ axis with $z$ taking its principal phase  $\arg z\in(0,2\pi)$. These are defined as \cite[(4.1), (4.2)]{Durand78}
\begin{eqnarray}
\label{LagLdef}
&&\hspace{-2.5cm}L_\gamma^\alpha(z):=\frac{\Gamma(\alpha+\gamma+1)}{\Gamma(\gamma+1)\Gamma(\alpha+1)}M(-\gamma,\alpha+1,z)=\frac{\Gamma(\alpha+\gamma+1)}{\Gamma(\gamma+1)\Gamma(\alpha+1)}\hyp11{-\gamma}{\alpha+1}{z},\\
&&\hspace{-2.5cm}
\label{LagNdef}
N_\gamma^\alpha(z):=\tfrac12\Gamma(\alpha+\gamma+1)\expe^zU(\alpha+\gamma+1,\alpha+1,-z),
\end{eqnarray}
where $-z:=\expe^{-i\pi}z$ and \cite[\href{http://dlmf.nist.gov/13.2.E42}{(13.2.42)}]{NIST:DLMF}
\begin{equation}
\hspace{0.5cm}U(a,b,z):=\frac{\Gamma(1-b)}{\Gamma(a-b+1)}\hyp11{a}{b}{z}+\frac{\Gamma(b-1)}{\Gamma(a)}z^{1-b}\hyp11{a-b+1}{2-b}{z},
\end{equation}
is the Kummer confluent hypergeometric function of the second kind. These correspond to confluent limits of the Jacobi functions, namely 
\cite[(4.3), (4.4)]{Durand78}
\begin{eqnarray}
&&\hspace{-3.5cm}L_\gamma^\alpha(z)= \lim_{\beta\rightarrow\infty} P_\gamma^{(\alpha,\beta)}\Big(1+\expe^{-i\pi}\frac{2z}{\beta}\Big),\quad N_\gamma^\alpha(z)= \lim_{\beta\rightarrow\infty} Q_\gamma^{(\alpha,\beta)}\Big(1+\expe^{-i\pi}\frac{2z}{\beta}\Big),
\end{eqnarray}
where the phase of $-2z/\beta=\expe^{-i\pi}2z/\beta$ is chosen to map values of $z=x+i0$, $z=x\expe^{2\pi i}$, on the upper (lower) sides of the cut in $N_\gamma^\alpha(z)$ to points $z-1=x\expe^{-i\pi}$ and $z-1=x\expe^{i\pi}$ on the cut in $Q_\gamma^{(\alpha,\beta)}(z)$ from 1 to $-\infty$.

\medskip
\noindent Finally, the Laguerre functions of the first and second kind on the positive real axis are defined for $x\in[0,\infty)$, as
\cite[(4.5), (4.6)]{Durand78}
\begin{eqnarray}
\label{LagLcutdef}
&&\hspace{-7.0cm}{\sf L}_\gamma^\alpha(x):=
\frac{i}{\pi}\left(\expe^{i\pi\alpha}N_\gamma^\alpha(x-i0)-\expe^{-i\pi\alpha}N_\gamma^\alpha(x+i0)\right),
\\
\label{LagNcutdef}
&&\hspace{-7.0cm}{\sf N}_\gamma^\alpha(x):=\tfrac{1}{2}\left(\expe^{i\pi\alpha}N_\gamma^\alpha(x-i0)+\expe^{-i\pi\alpha}N_\gamma^\alpha(x+i0)\right).
\end{eqnarray}


\medskip
\noindent We now derive a product formula for Laguerre functions of the second kind.

\begin{thm}
\label{thmNNprod}
Let $\gamma,\alpha\in\CC$, $\Re\alpha>0$, $\Re \gamma>-1$, $z_1,z_2\in\CC\setminus[0,\infty)$, $\arg z_1,\arg z_2\in(0,2\pi)$. Then
\begin{eqnarray}
&&\hspace{-1.0cm}N_\gamma^\alpha(z_1)N_\gamma^\alpha(z_2)=\frac{\sqrt{\pi}\,2^{\alpha-\frac12}\Gamma(\alpha+\gamma+1)}{\Gamma(\gamma+1)(z_1z_2)^{\frac{\alpha}{2}-\frac14}}\nonumber\\
&&\hspace{1.0cm}
\times \int_0^\infty
N_\gamma^\alpha(z_1+z_2+2\sqrt{z_1z_2}\cosh t)
\expe^{-\sqrt{z_1z_2}\cosh t}
I_{\alpha-\frac12}\left(\sqrt{z_1z_2}\sinh t\right)
(\sinh t)^{\alpha+\frac12}\,\,\dd t.
\end{eqnarray}
\end{thm}
\begin{proof}
This follows by making the replacements and change of variables in \eqref{JacQQQ} as described in \cite[\S4.2]{Durand78}, namely
$\{2y_1^2-1,2y_2^2-1,2y_3^2-1\}\mapsto\{\cosh{2t_1},\cosh{2t_2},\cosh{2t_1}\cosh{2t_2}+\sinh{2t_1}\sinh{2t_2}\cosh t\}$.
Upon making these replacements and change of integration variable from $z_3\to t$, the lower-limit of integration vanishes. This leads to the following result
\begin{eqnarray}
&&\hspace{0.0cm}Q_\gamma^{(\alpha,\beta)}(\cosh{2t_1})Q_\gamma^{(\alpha,\beta)}(\cosh{2t_2})=\frac{2^{\alpha-\frac{\beta}{2}}\expe^{-i\pi\beta}\Gamma(\alpha+\gamma+1)}{\Gamma(\gamma+1)\Gamma(\alpha+\beta)}\frac{(\coth{ t_1}\coth{ t_2)}^\alpha}{(\cosh{ t_1}\cosh{ t_2}^\beta}\nonumber\\
&&\hspace{0.5cm}\times
\int_0^\infty Q_\gamma^{(\alpha,\beta)}(\cosh{2t_1}\cosh{2t_2}+\sinh{2t_1}\sinh{2t_2}\cosh{ t})\,
(\sinh{ t})^\alpha \nonumber\\
&&\hspace{0.5cm}\times (1+\cosh{2t_1}\cosh{2t_2}+\sinh{2t_1}\sinh{2t_2}\cosh {t})^{\frac{\beta}{2}}Q_{\alpha-1}^{\beta}\left(\frac{\coth{t_1}\coth{t_2}+\cosh {t}}{\sinh{ }t}\right)\,\dd t.
\end{eqnarray}
Then we proceed as in \cite[\S4.2]{Durand78}, namely replacing
$\{\cosh{2t_1},\cosh{2t_2}\}\mapsto\{1+2\expe^{-i\pi}z_1/{\beta},1+2\expe^{-i\pi}z_2/{\beta}\}$,
and taking the limit as $\beta\to\infty$.
It is useful to write the associated Legendre function of the second kind in the integrand as a hypergeometric function using \cite[\href{http://dlmf.nist.gov/14.3.E7}{(14.3.7)}]{NIST:DLMF} and then we can use the confluent limit \cite[\href{http://dlmf.nist.gov/1.4.E3}{(1.4.3)}]{Koekoeketal} to write the Gauss hypergeometric function as a ${}_{0}F_1$ and then convert it to a modified Bessel function of the first kind using \cite[\href{http://dlmf.nist.gov/10.25.E2}{(10.25.2)}]{NIST:DLMF}. Alternatively and more directly, one can use the initial form of Lemma \ref{lemD}  in (\ref{lemm1}); the confluent limit of the the hypergeometric function in that expression gives the modified Bessel function  directly. 
Simplification completes the proof.
\end{proof}

\begin{rem} The result in Theorem \ref{thmNNprod} was given in a somewhat different form in \cite[(4.10)]{Durand78} in which the argument of the Laguerre function inside the integral was written as $(\sqrt{z_1}+\sqrt{z_2})^2+4\sqrt{z_1z_2}\sinh^2(t/2)$ to emphasize that there is a term that vanishes across the cut for $z_1=x$, $z_2=x\expe^{2\pi i}$. The present form is simpler.  We note also that $I_{\alpha-\frac12}(\sqrt{z_1z_2}\sinh{t})\big/(z_1z_2)^{\frac{\alpha}{2}-\frac{1}{4}}$ is an even function of $\sqrt{z_1z_2}$ so is continuous across the cut.
\end{rem}

\noindent The Nicholson-type formula given for Laguerre functions on the positive real axis as given in \cite[(4.12)]{Durand78} follows from Theorem \ref{thmNNprod} and the definitions (\ref{LagLcutdef}) and (\ref{LagNcutdef}) of the Laguerre functions on the cut by taking $z_1=x\expe^{i0}$ and $z_2=x\expe^{2\pi i}$ giving
\begin{eqnarray}
&&\hspace{-0.5cm}[{\sf N}_\gamma^\alpha(x)]^2+\frac{\pi^2}{4}[{\sf L}_\gamma^\alpha(x)]^2=N_\gamma^\alpha(x+i0)N_\gamma^\alpha(x-i0)\nonumber\\
&&\hspace{0.9cm}=\frac{\sqrt{\pi}\,\Gamma(\alpha+\gamma+1)2^{\alpha-\frac12}}{\Gamma(\gamma+1)x^{\alpha-\frac12}}\int_0^\infty
N_\gamma^\alpha(2x(1-\cosh t))
\expe^{x\cosh t}
I_{\alpha-\frac12}(x\sinh t)(\sinh t)^{\alpha+\frac12}\,\dd t,
\end{eqnarray}
where $x\in(0,\infty)$ and $\Re(\gamma+1)>0$, $\Re{\alpha}>0$
in agreement with \cite[(4.10)-(4.12)]{Durand78}.  We note that $N_\gamma^\alpha\propto z^{-\alpha-\gamma-1}\expe^z$ for $z\rightarrow\infty$ \cite[\href{http://dlmf.nist.gov/13.2.E6}{(13.2.6)}]{NIST:DLMF} so that the apparent exponential divergence in the factor $\expe^{x\cosh t}I_{\alpha-\frac12}(x\sinh t)$ is canceled by the behavior of $N_\gamma^\alpha$. The integral converges for $\Re{(\gamma+1)}>0$.


\subsection{Remark}

The above results can be extended to Bessel and Hermite functions by considering appropriate conformal limits. The results for Gegenbauer functions are described in \cite{Durand75}. We will not present those results here, but please note that they are also unaffected by the error in the lower limit of integration in the kernel form of the Jacobi $QQ$ product formula in that reference.  Some applications of the Nicholson-type formulas for Legendre and Gegenbauer functions are given in that reference, but the analysis there is clearly incomplete. We are not aware of any use of the results on Jacobi and Laguerre functions to obtain results on those functions analogous to those obtained on Legendre and Gegenbauer \cite[\S 5]{Durand75} or Bessel \cite[\S 13.73-13.75]{Watson},  where we note that the direct proof of the basic formula without the product formula is quite difficult. 

\subsection*{Acknowledgements}
We would like to thank Tom Koornwinder for valuable conversations. 



\def\cprime{$'$} \def\dbar{\leavevmode\hbox to 0pt{\hskip.2ex \accent"16\hss}d}

\end{document}